\def\NP{}
\documentclass[10pt]{amsart}

          \usepackage{amssymb}
          \usepackage{amsmath}
          \usepackage{amsthm}
          \usepackage{amsfonts}
          \usepackage{enumerate}
          \usepackage{url}
          \usepackage{tikz-cd}
          \usepackage{color}

\usepackage[numbers,sort&compress]{natbib}
\usepackage{epsfig}
\usepackage{ifthen}
\newcommand{\comment}[1]{}

\newcommand{\ind}{{\bf 1}}

\def\inddd#1{{\ind}_{\left\{#1\right\}}}
\def\indn#1{\{#1_n\}_{n\in\N}}
\def\inddds#1{{\ind}_{\left\{#1\right\}}^*}
\newcommand{\proba}{\mathbb P}

\newcommand{\esp}{{\mathbb E}}

\newcommand{\inv}{^{-1}}

\newcommand{\eqnh}{\begin{eqnarray*}}
\newcommand{\eqne}{\end{eqnarray*}}
\newcommand{\eqnhn}{\begin{eqnarray}}
\newcommand{\eqnen}{\end{eqnarray}}
\newcommand{\equh}{\begin{equation}}
\newcommand{\eque}{\end{equation}}

\def\summ#1#2#3{\sum_{#1 = #2}^{#3}}
\def\prodd#1#2#3{\prod_{#1 = #2}^{#3}}
\def\sif#1#2{\sum_{#1=#2}^\infty}

\newcommand{\eqd}{\stackrel{d}{=}}

\def\topp#1{^{(#1)}}

\def\abs#1{\left|#1\right|}

\def\ccbb#1{\left\{#1\right\}}

\def\pp#1{\left(#1\right)}

\def\mmid{\;\middle\vert\;}

\def\ceil#1{\left\lceil #1 \right\rceil}

\def\vv#1{{\boldsymbol #1}}

\def\qmand{\quad\mbox{ and }\quad}

\def\mwith{\mbox{ with }}
\def\qmwith{\quad\mbox{ with }\quad}
\def\mfa{\mbox{ for all }}

\def\mmas{\mbox{ as }}

\def\wt#1{\widetilde{#1}}
\def\wb#1{\overline{#1}}
\def\what#1{\widehat{#1}}

\def\limn{\lim_{n\to\infty}}

\def\limsupn{\limsup_{n\to\infty}}
\def\liminfn{\liminf_{n\to\infty}}
\def\weakto{\Rightarrow}

\def\Z{{\mathbb Z}}

\def\R{{\mathbb R}}

\def\N{{\mathbb N}}

\def\calA{\mathcal A}
\def\calB{\mathcal B}

\def\calF{\mathcal F}

\def\calG{\mathcal G}
\def\calJ{\mathcal J}
\def\calK{\mathcal K}
\def\calL{\mathcal L}
\def\calM{\mathcal M}

\def\calR{\mathcal R}

\def\calV{\mathcal V}
\def\calX{\mathcal X}

\def\topp#1{^{\scriptscriptstyle (#1)}}

\def\SM{{\rm SM}}
\def\USC{{\rm USC}}

\def\ddelta#1{\delta_{\pp{#1}}}

\def\CRSMa{\calM^{\rm Ca}_{\alpha,\calR}}
\def\CRSM{\calM^{\rm C}_{\alpha,\calR}}

\usepackage{epsfig}
\usepackage{ifthen}
\newtheorem{Thm}{Theorem}[section]
\newtheorem{Lem}[Thm]{Lemma}
\newtheorem{Prop}[Thm]{Proposition}

\theoremstyle{remark}
\newtheorem{Rem}[Thm]{Remark}

\newtheorem{Assump}[Thm]{Assumption}
\numberwithin{equation}{section}
\usepackage{epsfig}
\usepackage{ifthen}
\begin{document}\sloppy

% "Title of the Paper"
\title[Choquet random sup-measures with aggregations]{%A limit theorem for 
Choquet random sup-measures with aggregations}\date{\today}

\author{Yizao Wang}\address{Yizao Wang\\Department of Mathematical Sciences\\University of Cincinnati\\2815 Commons Way\\Cincinnati, OH, 45221-0025, USA.}\email{yizao.wang@uc.edu}

\begin{abstract}
A variation of Choquet random sup-measures is introduced. These random sup-measures are shown to arise as the scaling limits of empirical random sup-measures of a general aggregated model. 
Because of the aggregations, the finite-dimensional distributions of introduced random sup-measures do not necessarily have classical extreme-value distributions. 
Examples include the recently introduced stable-regenerative random sup-measures as a special case. 
\end{abstract}
\keywords{random sup-measure,random closed set, regular variation, point process, limit theorem, aggregated model}
\subjclass[2010]{Primary, 60G70, %extreme value theory
60F17; %functional limit theorems 
% 60G22; %fractional processes and fBm
  Secondary, 60G57. %random measure
  % 60C05. %combinatorial probability
   }

\maketitle

%\tableofcontents

%\tableofcontents
\NP

\section{Introduction}
\subsection{Background}For a general stationary sequence $\{X_i\}_{i\in\N}$, one is often interested in the asymptotic behavior of extremes. In the case that the random variables are i.i.d., it is well known that the global (macroscopic) asymptotic behavior of extremes is characterized by a two-dimensional Poisson point process.  
Assume that $\proba(X_1>x)= x^{-\alpha}L(x)$ with $\alpha>0$ and a slowly varying function $L$ at infinity (denoted by $\proba(X_1>x)\in RV_{-\alpha}$). It follows that for $a_n$ such that $\limn n\proba(X_1>a_n)= 1$ (implying $a_n\in RV_{1/\alpha}$), 
\equh\label{eq:PPP_convergence}
\summ i1n \ddelta{X_i/a_n,i/n}\weakto \sif \ell1\ddelta{\Gamma_\ell^{-1/\alpha},U_\ell},
\eque
in $\mathfrak M_p((0,\infty]\times[0,1])$, the space of Radon point measures on $(0,\infty]\times[0,1]$, where on the right-hand side $\{\Gamma_\ell\}_{\ell\in\N}$ are consecutive arrival times of a standard Poisson process, $\{U_\ell\}_{\ell\in\N}$ are i.i.d.~uniform random variables on $[0,1]$, and the two families are independent \citep{resnick87extreme}. 
Each pair $(\Gamma_\ell^{-1/\alpha},U_\ell)$ represent the magnitude and location of the $\ell$-th order statistic in the limit. 
We shall refer to every point in the limiting point process as an extreme for later discussions. We restrict the discussions to positive values only in introduction, for the sake of simplicity.

For a weakly dependent stationary sequence with the same marginal distribution as above, extremal clustering may occur, and with the same order of normalization a non-degenerate limit takes the form
\equh\label{eq:extremal_clustering}
\sif \ell1 \sum_{j\in\Z}\ddelta{\Gamma_\ell^{-1/\alpha}Q_{\ell,j},U_\ell}.
\eque
Here, for each $\ell\in\N$, $\{\Gamma_\ell^{-1/\alpha}Q_{\ell,j}\}_{j\in\Z}$ represents the magnitude of extremes belonging to the same cluster, and $\{Q_{\ell,j}\}_{j\in\Z},\ell\in\N$ are i.i.d.~copies of a sequence of non-negative random variables (it may have only a finite number of non-zero values), independent from $\{\Gamma_\ell^{-1/\alpha}\}_{\ell\in\N}$. 
Extremal clustering is a local feature here: it is referred to as the phenomena that if some $X_i$ from the original sequence takes a very large value, then with non-negligible probability, so are the few $X_j$ near $X_i$, which form a local cluster of extremes; when the time is scaled by $1/n$, the time indices of points in a cluster shrink to a single point in the limit. More precisely, convergence to clustering representation \eqref{eq:extremal_clustering} holds under the assumptions that $(X_1,\dots,X_k)$ has multivariate regular varying tails  with tail index $-\alpha$ for all $k\in\N$, the so-called anti-clustering assumption, and another mixing-type one. 
Many examples of stochastic processes and time series exhibiting extremal clustering are known. An extensive literature exists already on analysis for extremes of stationary sequences with weak dependence (e.g.~\citep{davis95point,davis98sample,basrak09regularly,basrak18invariance,dombry18tail,kulik20heavy}). 

For both cases above, the macroscopic behaviors of the extremes are the same, if one looks at the limit of {\em random sup-measures}. 
Recall that a sup-measure $m$ on $[0,1]$ is a set function on subsets of $[0,1]$, satisfying the relation $m(\bigcup_\lambda A_\lambda) = \sup_\lambda m(A_\lambda)$ for $A_\lambda\subset[0,1]$, and the law of a random sup-measure, say $\calM$, is uniquely determined by its finite-dimensional distributions over a suitable class $\calA$ of subsets of $[0,1]$ (e.g.~all open subsets). 
The convergence of empirical random sup-measures takes the form
\equh\label{eq:RSM_convergence}
\frac1{a_n}M_n(\cdot):=\frac1{a_n}\max_{k/n\in\cdot}X_k \weakto \calM(\cdot),
\eque
in the space of sup-measures.  
In the i.i.d.~case discussed earlier, with the same choice of $a_n$ as in \eqref{eq:PPP_convergence} 
the limit on the right-hand side above is an independently scattered $\alpha$-Fr\'echet random sup-measure on $[0,1]$ with Lebesgue control measure, denoted by $\calM^{\rm is}_\alpha$. This random sup-measure has the representation
\[%\equh\label{eq:Mis}
\calM_\alpha^{\rm is}(\cdot) \eqd \sup_{\ell\in\N}\frac1{\Gamma_\ell^{1/\alpha}}\inddd{U_\ell\in\cdot},
\]%\eque
sharing the same Poisson point process as on the right-hand side of \eqref{eq:PPP_convergence}. 
Alternatively, the law of $\calM^{\rm is}_\alpha$ is determined by the properties that it has $\alpha$-Fr\'echet marginal distribution ($\proba(\calM_\alpha^{\rm is}(A)\le x) = e^{-{\rm Leb}(A)x^{-\alpha}}$ for all open $A$ and $x>0$) and that it is independently scattered ($\{\calM_\alpha^{\rm is}(A_i)\}_{i=1,\dots,d}$ are independent for any disjoint collection $\{A_i\}_{i=1,\dots,d}$).  
%It is worth mentioning that when the limit random sup-measure is independently scattered, the point-process approach provides a stronger result, as the random-sup-measure approach does not characterize local clustering. %The convergence of the random sup-measure   \eqref{eq:RSM_convergence} with the limit as in \eqref{eq:Mis}) follows from the point-process convergence \eqref{eq:PPP_convergence} immediately (actually, this holds more generally for models with possibly long-range dependence, e.g.~\citep[proof of Theorem 4.2]{durieu18family}). Moreover, it is easy to see that when {extremal clustering} appears, the limit point process in \eqref{eq:extremal_clustering} corresponds to the same $\calM^{\rm is}_\alpha$ up to a multiplicative constant. So in the presence of weak dependence, the point-process convergence may reveal local structure that the random sup-measure is unable to capture.

In the presence of strong dependence, the extremes of a stationary sequence may have a completely different behavior at macroscopic level.
In a seminal work, \citet{obrien90stationary} characterized all possible limits of extremes that may arise and advocated the necessity of using  random sup-measures to characterize macroscopic extremes of random variables with strong dependence. However, not many representative examples of limit theorems were immediately known, and the paper \citep{obrien90stationary} did not attract enough attention until recently. 

A notable family of random sup-measures, the $\alpha$-Fr\'echet {\em Choquet random sup-measures} (CRSMs), appeared in recent investigations following \citep{obrien90stationary}. This family of $\alpha$-Fr\'echet random sup-measures were introduced by \citet{molchanov16max}. Recall that a random sup-measure $\calM$ is $\alpha$-Fr\'echet if $(\calM(A_1),\dots,\calM(A_d))$ has a multivariate $\alpha$-Fr\'echet distribution for all $A_1,\dots,A_d$.
An $\alpha$-Fr\'echet CRSM on $[0,1]$ has in addition 
the following representation:
\equh\label{eq:CRSM}
\calM_{\alpha,\calR}^{\rm C}(\cdot) \eqd \sup_{\ell\in\N}\frac1{\Gamma_\ell^{1/\alpha}}\inddd{\calR_\ell\cap\cdot\ne\emptyset},
\eque
where $\{\Gamma_\ell\}_{\ell\in\N}$ are as before, $\{\calR_\ell\}_{\ell\in\N}$ are i.i.d.~copies of a random closed set $\calR$ taking values from $\calF([0,1])$ (the space of closed subsets of $[0,1]$), and the two families are independent. 
Then, in view of the magnitude-location interpretation, 
the extremes with magnitude $\Gamma^{-1/\alpha}_\ell$ may appear at multiple (possibly infinite) locations, recorded in $\calR_\ell$. These statistics are summarized in the Poisson point process
\equh\label{eq:Choquet}
\sif \ell1 \ddelta{\Gamma_\ell^{-1/\alpha},\calR_\ell},
\eque
this time in $\mathfrak M_p((0,\infty]\times\calF([0,1]))$.
We shall refer to the random sup-measure \eqref{eq:CRSM} as an $(\alpha,\calR)$-CRSM (with a little abuse of notation, it depends only on the law of $\calR$ instead of a random closed set on some probability space). 
Note that CRSM includes independently scattered random sup-measures as a special case by taking $\calR_\ell = \{U_\ell\}$ with $U_\ell$ as in the previous representation. In the case that each $\calR_\ell$ has finitely many points with probability one, denoted by $\calR_\ell=\{V_{\ell,j}\}_j$, \eqref{eq:Choquet} has the following counterpart in $\mathfrak M_p((0,\infty]\times[0,1])$,
\equh\label{eq:Choquet1}
\sif \ell1\sum_j\ddelta{\Gamma_\ell^{-1/\alpha},V_{\ell,j}}.
\eque
In contrast to \eqref{eq:extremal_clustering}, this time $\{(\Gamma_\ell^{-1/\alpha},V_{\ell,j})\}_j$ form a cluster of extremes that appear at different macroscopic time locations. When $\calR$ has infinitely many points, such a point process \eqref{eq:Choquet1} will have infinitely many points in compact intervals and hence it is not Radon, and standard tools for point-process convergence do not apply. In this case it is natural to work with \eqref{eq:Choquet} with limit theorems stated in terms of random sup-measures as in \eqref{eq:RSM_convergence}. It is worth keeping in mind that when the limit of extremes can be represented as a point process as \eqref{eq:Choquet1}, the point-process approach should be preferred to the random-sup-measure one. The main reason to work with random sup-measures here is to be able to deal with examples with $\calR$ as random fractals.  See Remark \ref{rem:comparison} for more discussions.

\subsection{Overview of main results}
The phenomena that the same magnitude may appear at multiple locations in the macroscopic limit is referred to as {\em long-range clustering}, for which CRSMs provide a natural framework (unless $\calR$ is a random singleton). Limit theorems for CRSMs  have just appeared recently in \citep{durieu18family,lacaux16time}. 
We continue the recent investigations on CRSMs, and our contribution is twofold.
\begin{enumerate}[(a)]
\item First,  we introduce a family of random sup-measures that can be viewed as a variation of the CRSMs. We refer to this new family as {\em Choquet random sup-measures with aggregations}. 
\item Second, we provide a general aggregated model, of which the empirical random sup-measure scales to a CRSM, possibly with aggregations. 
\end{enumerate}

The CRSM with aggregations has  the following representation. Let $\{(\Gamma_\ell,\calR_\ell)\}_{\ell\in\N}$ be as before, and write $\calR_J := \bigcap_{j\in J}\calR_j$. Then, 
\equh\label{eq:CRSMa}
\calM_{\alpha,\calR}^{\rm Ca}(\cdot) \eqd 
\sup_{
J\subset \N, \calR_J\cap\cdot\ne\emptyset}\sum_{j\in J}\frac1{\Gamma_j^{1/\alpha}}, \eque
is referred to as an $(\alpha,\calR)$-CRSM with aggregations. (Note that $\calR_J \ne\emptyset$ is allowed with $J\subset\N$ having infinite cardinality, as a random sup-measure may take $+\infty$ value. However we shall exclude such a possibility later. See Remark \ref{rem:assump}, \eqref{item:2}.) In this case the point process
\equh\label{eq:CRSMa_PP}
\sum_{J\subset\N:\calR_J\ne\emptyset}
\ddelta{\sum_{j\in J}\Gamma_j^{-1/\alpha},\calR_J}
\eque
records all the magnitudes and locations of extremes. Aggregation here refers to the fact that at intersections, the extreme value is the sum of more than one value of $\Gamma_j^{-1/\alpha}$. 
If the random closed sets do not intersect with probability one, then $\CRSM \eqd\CRSMa$. The interesting case is when they intersect with strictly positive probability. In this case, $\CRSMa$ is no longer $\alpha$-Fr\'echet. In addition, we shall restrict to the case that $\calR$ is light (i.e., for every fixed $t$, $\proba(t\in\calR) = 0$). For the case that $\calR$ is not light, the CRSM with aggregations is of a different nature, and will not be considered here.

After the introduction of CRSM with aggregations, in the second part of the paper, we shall provide a simple aggregation framework that leads to these random sup-measures in the limit. 
We are interested in random sup-measures on a subset $E\subset[0,1]$ (we shall have examples with $E = (0,1]$ and $(0,1)$). 
Let $\indn X$ be i.i.d.~random variables with $\proba(X_1>x)\sim \mathsf p\proba(|X_1|>x)\in RV_{-\alpha}$ for some $\mathsf p\in(0,1]$.  Let $\indn R$ be a sequence of random closed sets such that 
$R_n\weakto \calR$
in $\calF(E)$ ($\weakto$ as convergence in distribution for random elements in the corresponding metric space \citep{billingsley99convergence}). 
For each $n$, let $\{R_{n,j}\}_{j\in\N}$ be i.i.d.~copies of $R_n$ taking values from $\{0,\dots,n\}/n$, independent from $\{X_j\}_{j\in\N}$. 
We shall consider a triangular-array model, and show that
the following convergence of empirical random sup-measure
\[%\equh\label{eq:Mmn}
 \frac1{a_n}\max_{k/n\in \cdot}\summ j1{m_n}X_j\inddd{k/n\in R_{n,j}}\weakto \calM_{\alpha,\calR}^{\rm Ca}(\cdot),
\]%\eque
 as $n\to\infty$, for appropriately chosen $a_n$ and $m_n$. The key assumptions leading to the desired convergence are that $\calR$ is light and that
\[%\equh\label{eq:RnJ}
\bigcap_{j\in J}R_{n,j}\weakto \bigcap_{j\in J}\calR_j \mbox{ for every $J\subset\N$ fixed.}
\]%\eque 
All other assumptions are mild. 
Our examples include old and new CRSMs with and without aggregations.

\subsection{Comments}
We conclude the introduction with a few comments on our results. 
\begin{Rem}The motivating example for us came from the recent paper \citep{samorodnitsky19extremal}. Therein, the first example of CRSMs with aggregations was introduced,  where $\calR$ is a {\em randomly shifted $\beta$-stable regenerative set}, denoted as $\calR_\beta^{\rm srs}$ here and to be recalled in Section \ref{sec:srs}. Non-trivial intersections of  independent random closed sets occur and hence the CRSM is with aggregation.  
It was shown in \citep{samorodnitsky19extremal} that $\calM^{\rm Ca}_{\alpha,\calR_\beta^{\rm srs}}$ arises as the scaling limit of the empirical random sup-measure of a family of stationary processes with regularly-varying tails and long-range dependence, first introduced in \citep{rosinski96classes}. 
The long-range dependence in this example is intrinsically related to the underlying randomly shifted stable-regenerative sets $\calR_\beta^{\rm srs}$: such random closed sets (their corresponding local times resp.) arise recently in extremal (central resp.) limit theorems for stochastic processes introduced in \citep{rosinski96classes} and their variations \citep{owada15functional,lacaux16time,bai20functional}. A new representation of Hermite processes has also been introduced recently based on stable-regenerative sets \citep{bai20representations}. 

The original motivation of this paper was to understand the underlying mechanism of the process in \citep{samorodnitsky19extremal} that leads this remarkable family of random sup-measures $\calM_{\alpha,\calR_\beta^{\rm srs}}^{\rm Ca}$. Remark \ref{rem:motivation} explains how our limit theorem sheds light on the underlying dynamics with long-range dependence.
\end{Rem}

\begin{Rem}
In Section \ref{sec:srs} a couple examples of limit theorems for CRSMs with aggregations are provided, and both are essentially related to stable-regenerative sets. The framework proposed in this paper is a little unsatisfactory in the sense that we are unaware of any other examples leading to CRSMs with aggregations where $\calR$ is of {\em a different type of random fractals}. It is an interesting question how, and what type of, other random fractals may arise naturally from a discrete-time stochastic  model in general, and the answer to this question should lead to new examples of CRSMs with aggregations.
\end{Rem}

\begin{Rem}\label{rem:comparison}
Random sup-measures are convenient at providing a unifying framework for limit theorems for extremes. In order to include the example with $\calR$ as a stable-regenerative set, which is the most interesting to us, we choose to use random closed sets representing the locations of points from each cluster (and we have explained why \eqref{eq:Choquet1} would not work), and then more generally to state limit theorems in terms of random sup-measures. Note that the point-process convergence still plays a crucial role; see Proposition \ref{prop:top} and Lemma \ref{lem:top_PP}.

On the other hand, when $\calR$ consists of finite number of points (e.g.~the Karlin random sup-measures introduced recently in \citep{durieu18family}, see Section \ref{sec:Karlin}), working with random sup-measures is unnecessary and actually may provide less information on the extremes. In particular, one should keep in mind that while random sup-measures are useful at characterizing global dependence structures of extremes, they do not reveal any local clustering dependence structure. Therefore when restricted to such a case, working with point processes as in \eqref{eq:Choquet1} (or more sophisticated formulations of point processes for local clustering, e.g.~\citep{basrak09regularly,kulik20heavy})  yields stronger results in principle.
\end{Rem}
\begin{Rem}It is well known that aggregation schemes are one of the main resources that lead to long-range dependence \citep{samorodnitsky16stochastic,pipiras17long}.  
Our model has a flavor of random walks in random sceneries (see e.g.~\citep{cohen06random,dombry09discrete,treszczotko18random,cohen09convergence}).
Similar aggregated models have also been investigated in queueing theory with long-range dependence \citep{kaj08convergence,mikosch07scaling}. However, very few references can be found in the literature on extremes of aggregated models of the type considered here. 
\end{Rem}

The paper is organized as follows. Section \ref{sec:RSM} introduces CRSM with aggregations. Section \ref{sec:model} introduces the aggregated model and shows that its empirical random sup-measures scales to the CRSM with aggregations. Various examples are provided in Section \ref{sec:example}. 

\NP
\section{CRSM with aggregations} \label{sec:RSM}
Our references for random sup-measures are \citet{obrien90stationary,vervaat97random}, and for random closed sets  \citet{molchanov17theory}. 
 We only provide the minimum background in Section \ref{sec:background} without full details.
We then provide three representations for CRSM with aggregations in Section \ref{sec:three}. The results of this section are straightforward consequences of standard background. 

\subsection{Background}\label{sec:background}
Throughout, consider $E\subset[0,1]$ such that $[0,1]\setminus E$ has at most finite number of points (or empty), with the topology induced by the Euclidean metric.  For all our examples in Section \ref{sec:example}, $E = [0,1], (0,1]$ or $(0,1)$.  
Let $\calF = \calF(E), \calG = \calG(E),\calK = \calK(E)$ denote the collection of closed, open and compact subsets of $E$, respectively. 

Let $\calF(E)$ be equipped with the Fell topology. A basis of the Fell topology is $\calF^K_{G_1,\dots,G_d} = \{F\in\calF(E): F\cap K = \emptyset, F\cap G_i \ne\emptyset\}$ for $K\in\calK$ and $d\in\N,  G_1,\dots,G_d\in\calG$. The Fell topology is metrizable for locally compact second countable Hausdorff and hence our choice of $E$. The Borel $\sigma$-algebra $\calB(\calF(E))$ of $\calF(E)$ is generated by the sets $\calF_K = \{F\in\calF(E):F\cap K\ne\emptyset$ for all $K\in\calK$. Random closed sets are measurable mappings from certain probability space to  $(\calF(E),\calB(\calF(E)))$. The law of a random closed set is determined by the capacity functional evaluated at all finite unions of elements from a {\em separating class}, for which one may take 
\equh\label{eq:G_0}
\calG_0:=\ccbb{(a,b):[a,b]\subset E}.
\eque
That is, the law is a random closed set $\calR$ is uniquely determined by the probabilities $\proba(\calR\cap A_i\ne\emptyset, i=1,\dots,d)$ for all $d\in\N,A_1,\dots,A_d\in\calG_0$. If $\{R_n\}_{n\in\N}$ and $\calR$ are random closed sets in $\calF(E)$ that in addition $\proba(\calR\cap A\ne \emptyset) = \proba(\calR\cap \wb A\ne\emptyset)$ for all $A\in\calG_0$, then $R_n\weakto \calR$ if 
$\limn\proba(R_n\cap A_i\ne\emptyset, i=1,\dots,d) = \proba(\calR\cap A_i\ne\emptyset,i=1,\dots,d)$ for all $d\in\N, A_1,\dots,A_d\in\calG_0$ \citep[Corollary 1.7.14]{molchanov17theory}.

A sup-measure $m$ on $E$ taking values in $\wb\R = [-\infty,\infty]$ is a set function on all subsets of $E$, satisfying 
\[%\equh\label{eq:sup_measure}
m\pp{\bigcup_{\lambda\in\Lambda}A_\lambda} = \sup_{\lambda\in\Lambda}m(A_\lambda), 
\]%\eque
for all collections of subsets $\{A_\lambda\}_{\lambda\in\Lambda}$ of $E$, and $m(\emptyset) = -\infty$. We let $\SM\equiv \SM(E,\wb\R)$ denote the space of all $\wb\R$-valued sup-measures on $E$. 
Again, for our choice of $E$ every $m\in\SM$ is uniquely determined by the values of $m$ over a subbase of the topology of $E$, for example by $\{m(G)\}_{G\in\calG_0}$. 
The space $\SM$ is equipped with the so-called sup vague topology. In this topology, $m_n\to m$ as $n\to\infty$, if
\begin{align*}
\limsupn m_n(K) & \le m(K), \mfa K\in\calK,\\
\liminfn m_n(G) & \ge m(G), \mfa G\in\calG.
\end{align*}
With the sup vague topology, $\SM$ is compact and Polish \citep{norberg86random}. So a random sup-measure is a measurable mapping from certain probability space to $(\SM,\calB(\SM))$ ($\calB(\SM)$ as the Borel $\sigma$-algebra of $\SM$). In particular, two random sup-measures are equal in distribution if they have the same finite-dimensional distributions when evaluated over a {\em probability-determining class}, of which $\calG_0$ in \eqref{eq:G_0} is an example \citep[Theorem 11.5]{vervaat97random}. 
That is, for a random sup-measure $M(\cdot)$, one may simply view it as a set-indexed stochastic process, the law of which is uniquely determined by $\{M(G)\}_{G\in\calG_0}$ with $\calG_0$ in \eqref{eq:G_0}.

We need two notions of convergence of random sup-measures. The first is the pointwise convergence. Let $\{M_n\}_{n\in\N}$ be a family of increasing random sup-measures as measurable mappings from certain $(\Omega,\calA,\proba)$ to $(\SM, \calB(\SM))$. 
We write $M_n(G)$ as the random sup-measure evaluated at $G$ and $M_n(\omega,G)$ when we want to emphasize the dependence on $\omega\in\Omega$. Here, by an increasing family we mean that for all $m\le n, G\in\calG, \omega\in\Omega$, $M_m(\omega,G)\le M_n(\omega,G)$. We shall then understand
\[
\sup_{n\in\N}M_n
\]
as a random element in SM (a measurable mapping from $(\Omega,\calA,\proba)$ to $\SM$) defined as follows. 
For {every $\omega\in\Omega$ fixed}, $M(\omega,G):=\sup_{n\in\N} M_n(\omega,G), G\in\calG$ defines a sup-measure $M(\omega)\equiv M(\omega,\cdot)\in\SM$, and $M_n(\omega)\to M(\omega)$ in $\SM$ (with respect to the sup vague topology) as $n\to\infty$ \citep[Theorem 6.2]{vervaat97random}. Since the space $\SM$ is metrizable and every $M_n$ is measurable, it follows that $\omega\mapsto M(\omega)$ is measurable \citep[Lemma 1.10]{kallenberg97foundations}. We shall set $\sup_{n\in\N}M_n:= M$ as the pointwise limit.
Second, a sequence of random sup-measures $\{M_n\}_{n\in\N}$ {\em converges in distribution} to another random sup-measure $M$ in $\SM$, if 
\[%\equh\label{eq:RSM_convergence}
\pp{M_n(G_1),\dots,M_n(G_d)}\weakto \pp{M(G_1),\dots,M(G_d)}
\]%\eque
for all $d\in\N$ and $G_1,\dots,G_d$ from any {\em convergence-determining class} (again we may take $\calG_0$) such that 
$M(G_i) = M(\wb G_i), i=1,\dots,d$ almost surely  \citep[Theorem 12.3]{vervaat97random}.

%\NP
\subsection{Three representations}\label{sec:three}
We provide three different representations for $(\alpha,\calR)$-CRSM with aggregations. Throughout, $\alpha>0$, $\calR$ is a random closed set taking values from $\calF(E)$ and  $\{\calR_\ell\}_{\ell\in\N}$ are i.i.d.~copies of $\calR$, independent from $\{\Gamma_j\}_{j\in\N}$, which always represents a collection of consecutive arrival times of a standard Poisson process. 
Introduce
\equh\label{eq:M_J}
\calM_J(\cdot)\equiv \calM_{\alpha,\calR,J}(\cdot):= 
\begin{cases}
\displaystyle \sum_{j\in J}\frac1{\Gamma_j^{1/\alpha}} & \mbox{ if } \calR_J\cap\cdot\ne\emptyset,\\
-\infty & \mbox{ otherwise,}
\end{cases}
%\qmwith \calR_J:=\bigcap_{j\in J}\calR_j,\quad J\subset\N, |J|<\infty,
\eque
with $\calR_J:=\bigcap_{j\in J}\calR_j, J\subset\N, |J|<\infty$,
and $\calM_{\emptyset}(\cdot) \equiv -\infty$. In this section, except for $\CRSMa$ we do not write explicitly the dependence of random sup-measures on $\alpha,\calR$ most of the time.
{The $(\alpha,\calR)$-CRSM with aggregations} is defined as
\begin{equation}\label{eq:RSM2}
 \CRSMa:=\sup_{J\subset\N,|J|<\infty}\calM_{J}.%, \quad \mbox{ almost surely.}
\end{equation}
We shall assume that almost surely, $\calR_J = \emptyset$  (hence $\calM_J\equiv -\infty$) for $|J|$ large enough. This assumption is a consequence of our Assumption \ref{assump:Rn} later; see Remark \ref{rem:assump}, (ii).  
Here, the supremum in \eqref{eq:RSM2} is understood as the pointwise limit 
of $\calM_\ell:=\max_{J\subset[\ell]}\calM_J$
as $\ell\to\infty$. 
Throughout, we write $[\ell] = \{1,\dots,\ell\}$. 
\begin{Rem}\label{rem:RSM}
It might be convenient to introduce the following special indicator function,  for any event $A$, 
\equh\label{eq:indicator*}
{\bf 1}^*_A:=\begin{cases}
\displaystyle 1 & \mbox{on the event $A$} \\
 -\infty & \mbox{otherwise.}
\end{cases}
\eque
In this way, one may write $\calM_J = (\sum_{j\in J}\Gamma_j^{-1/\alpha})\inddds{\calR_J\cap\cdot\ne\emptyset}$ in \eqref{eq:M_J}
 with the convention $0\cdot(-\infty) = -\infty$. 

Allowing random sup-measures to take $-\infty$ may cause some confusion for the first time. If one is interested in $\CRSMa$ alone but not the limit theorem, then it suffices to consider its representations in $\SM(E,\wb\R_+)$, which is simpler by using the ordinary indicator functions: $\CRSMa(\cdot) = \sup_{J\subset\N, |J|<\infty}(\sum_{j\in J}\Gamma_j^{-1/\alpha})\inddd{\calR_J\cap\cdot\ne\emptyset}$.  However, we need to work with $\SM(E,\wb \R)$ for our limit theorems later (Proposition \ref{prop:RSM3} and Theorem \ref{thm:1}),
as random sup-measures of interest may take negative values. 
\end{Rem}

The second representation plays a crucial role in our limit theorem. 
In the sequel, let $\{\varepsilon_\ell\}_{\ell\in\N}$ be i.i.d.~random variables with  
\[
\proba(\varepsilon_1 = 1) =\mathsf p = 1-\proba(\varepsilon_1= -1), \mbox{ for some } \mathsf p\in(0,1],
\]
 independent from $\{(\Gamma_\ell,\calR_\ell)\}_{\ell\in\N}$.  
Introduce
\begin{multline}\label{eq:M_l,epsi}
\calM_\ell\topp {\mathsf p}(\cdot):= \max_{J\subset[\ell]}\calM_{J}\topp{\mathsf p}(\cdot)\\\qmwith \calM_{J}\topp{\mathsf p}(\cdot):= 
\begin{cases}
\displaystyle\sum_{j\in J}\frac{\varepsilon_j}{\Gamma_j^{1/\alpha}} & \mbox{ if } \calR_J\cap\cdot\ne\emptyset,\\
-\infty & \mbox{ otherwise,}
\end{cases}
J\subset\N, |J|<\infty.
\end{multline}
\begin{Prop}\label{prop:RSM3}With $\CRSMa$ as in \eqref{eq:RSM2} and $\calM_\ell\topp{\mathsf p},\calM_{J}\topp{\mathsf p}$ in \eqref{eq:M_l,epsi}, we have
\[%\equh\label{eq:RSM3}
\sup_{J\subset\N, |J|<\infty}\calM_{J}\topp{\mathsf p} \eqd \mathsf p^{1/\alpha} \CRSMa \qmand \calM_\ell\topp{\mathsf p}\weakto \mathsf p^{1/\alpha} \CRSMa
\]%\eque
as $\ell\to\infty$.
\end{Prop}
\begin{proof}
Since $\sup_{J\subset\N,|J|<\infty}\calM_J\topp{\mathsf p}$ is the pointwise limit of $\calM_\ell\topp{\mathsf p}$ as $\ell\to\infty$, it suffices to prove the first part. We first introduce 
\[
\calM_J\topp+ := \begin{cases}
\calM_J\topp{\mathsf p} & \mbox{if } \varepsilon_j = 1, \forall j\in J,\\
-\infty & \mbox{otherwise,}
\end{cases}
\]
and $\calM_{\emptyset}\topp+ \equiv-\infty$.  
 Then, it is clear that
\equh\label{eq:thinning_M}
\sup_{J\subset\N,|J|<\infty}\calM_J\topp+ \eqd  \mathsf p^{1/\alpha}\sup_{J\subset\N,|J|<\infty}\calM_J.
\eque
Indeed, this follows from the thinning property that
\[
\ccbb{\pp{\Gamma_j^{-1/\alpha},\calR_j}}_{j\in\N, \varepsilon_j = 1}\eqd \ccbb{\pp{\mathsf p^{1/\alpha}\Gamma_j^{-1/\alpha},\calR_j}}_{j\in\N},
\]
and the observation that each side of \eqref{eq:thinning_M} is the same deterministic functional of the Poisson point process on the corresponding side above. 

It remains to show
\equh\label{eq:thinning_M2}
\sup_{J\subset\N,|J|<\infty}\calM_{J}\topp{\mathsf p} = \sup_{J\subset\N,|J|<\infty}\calM_J\topp+, \mbox{ almost surely.}
\eque
Recall that both sides are understood as limits of increasing random sup-measures, which are defined for all $\omega\in\Omega$. Now, we shall show that both sides coincide on a countable family of open intervals, and for this purpose it suffices to consider 
\equh\label{eq:GQ}
\calG_0^{\mathbb Q}:=\{(a,b)\in\calG_0:a,b\in\mathbb Q\},
\eque and  prove that the two random sup-measures of interest take the same value for every $G\in\calG_0^{\mathbb Q}$ almost surely. 
Observe that if $\proba(\calR\cap G\ne\emptyset) = 0$, then almost surely both sides above are $-\infty$. So consider $G$ such that $\proba(\calR\cap G\ne\emptyset)>0$, then both sides are almost surely non-negative. 
For any finite $J\subset\N$ introduce $J_+\equiv J_+(\{\varepsilon_j\}_{j\in J}):=\{j\in J:\varepsilon_j = 1\}$. 
Then $\calM_{J}\topp{\mathsf p}(G)\le \calM_J\topp+(G)$, with the equality holds if and only if $J_+ = J$. 
So the inequality `$\le$' of \eqref{eq:thinning_M2} is trivial.
For the inequality `$\ge$', it is clear that $\calM_{J_+}\topp{\mathsf p} = \calM_J\topp+$. This completes the proof.
\end{proof}

The last representation makes use of upper-semi-continuous functions. 
While such representations are commonly used in the literature, they are not as convenient when working with CRSMs with aggregations, and we do not need them for the proof.  
Recall that for a sup-measure $m\in \SM(E,\wb \R)$, it is related to the corresponding upper-semi-continuous function $f$ by the sup-derivative and sup-integral operators $d^\vee$ and $i^\vee$, respectively,
\begin{align*}
d^\vee: \SM(E,\wb \R) \to \USC (E,\wb\R)& \qmwith (d^\vee m)(t) = \inf_{G\in\calG, G\ni t}m(G),\\
i^\vee: \USC(E,\wb \R) \to \SM (E,\wb\R)& \qmwith (i^\vee f)(G) = \sup_{t\in E,t\in G} f(t),
\end{align*}
where $\USC(E,\wb\R)$ denote the space of upper-semi-continuous functions on $E$ taking values in $\wb \R$. 
The two operators are the inverse of each other and the induced relation is a  homeomorphism between $\SM(E,\wb\R)$ and $\USC(E,\wb\R)$. 
Introduce
\[
\calJ_t\equiv \calJ_t(\{\calR_\ell\}_{\ell\in\N}):= \ccbb{\ell\in\N:t\in \calR_\ell}, \quad \xi_\alpha(t):=
\begin{cases}
\displaystyle \sum_{j\in \calJ_t}\frac1{\Gamma_j^{1/\alpha}} & \mbox{ if } \calJ_t\ne\emptyset,\\
-\infty & \mbox{ otherwise,}
\end{cases}
~ t\in E.
\]
(Using \eqref{eq:indicator*}, one can write $\xi_\alpha(t) = (\sum_{j\in\calJ_t}\Gamma_j^{-1/\alpha})\inddds{\calJ_t\ne\emptyset}$.)
It is straightforward to verify that $\xi_\alpha\in\USC(E,\wb\R)$ almost surely.
\begin{Prop}%\label{RSM2} 
With notations above,
\equh\label{eq:RSM}
\CRSMa(\cdot) = \pp{i^\vee \xi_\alpha}(\cdot) \equiv \sup_{t\in\cdot}\xi_\alpha(t) \mbox{ almost surely.}
\eque
\end{Prop}

\begin{proof}
Let $\calM$ denote the right-hand side of \eqref{eq:RSM}. It suffices prove, 
\[%\equh\label{eq:MG}
\calM(G) = \CRSMa(G)=\sup_{J\subset N,|J|<\infty}\calM_J(G) \mbox{ almost surely,} \mfa G\in\calG_0^{\mathbb Q}.%, \mbox{ almost surely.}
\]%\eque
Fix $G\in\calG_0^{\mathbb Q}$ as in \eqref{eq:GQ}. Clearly, for every $t\in G$, $\xi_\alpha(t)\le \CRSMa(G)$, whence $\calM(G)\le \CRSMa(G)$. On the other hand, for every $J\subset\N$, $|J|<\infty$ such that $\calM_J(G)\ne -\infty$, there exists $t\in G$ such that $t\in \calR_J$ and $\calM_J(G) = \sum_{j\in J}\Gamma_j^{-1/\alpha}$. But $t\in \calR_J$ also implies that $\sum_{j\in J}\Gamma_j^{-1/\alpha} \le\xi_\alpha(t)$. Thus $\calM(G)\ge \CRSMa(G)$. This completes the proof. 
\end{proof}

\NP
\section{A limit theorem for an aggregation framework}\label{sec:model}
We introduce a general aggregated model, of which the empirical random sup-measure scales to an $(\alpha,\calR)$-CRSM with aggregations. 
Let $\{X_j\}_{j\in\N}$ be i.i.d.~random variables and $\{R_{n,j}\}_{j\in\N}$ be i.i.d.~copies of certain random closed set $R_n$. 
 Write $R_{n,J}:=\bigcap_{j\in J} R_{n,j}$. Recall notations for $\{\calR_j\}_{j\in\N}$ and $\calR_J$ from the previous section \eqref{eq:M_J}.
Set
\[
p_n(k) := \proba(k/n \in R_n), k=0,\dots,n \qmand p_n(G) := \max_{k/n\in G}p_n(k).
\]
Throughout, $C$ denotes a strictly positive constant that may change from line to line.
We start with assumptions on $R_{n}$ and $\calR$. 
\begin{Assump}\label{assump:Rn}
Let $E$ be a subset of $[0,1]$ such that $[0,1]\setminus E$ has at most finite number of points. 
For each $n\in\N$,
\equh\label{eq:lattice}
R_n\subset\frac1n\{0,\dots,n\} \mbox{ almost surely.}
\eque
For all $\ell\in\N$,
\equh\label{eq:intersection_convergence}
R_{n,[\ell]}\weakto \calR_{[\ell]} \mbox{ in } \calF(E).
\eque
There exists $c_0>0$ such that for all $G\in\calG_0\equiv \calG_0(E)$, 
\equh\label{eq:p_n}
p_n(G)\le C n^{-c_0} \mfa n\in\N,
\eque
where the constant $C$ may depend on $G$.
\end{Assump}

Next, we assume
\equh\label{eq:RV}
\proba(|X_1|>x) = x^{-\alpha}L(x), \alpha>0 \qmand \lim_{x\to\infty}\frac{\proba(X_1>x)}{\proba(|X_1|>x)} = \mathsf p 
%\qmand \lim_{x\to\infty}\frac{\proba(X_j<-x)}{\proba(|X_j|>x)} = 1-\mathsf p 
\mbox{ for some $\mathsf p\in(0,1]$},
\eque
where $L$ is a slowly varying function at infinity. 
Let $\{m_n\}_{n\in\N}$ denote an increasing sequence of positive integers, and 
\[
J_{n,k}:= \ccbb{j=1,\dots,m_n: k/n\in R_{n,j}}.
\]
The empirical random sup-measure of the aggregated model then is, 
\[%\equh\label{eq:M}
M_{n}(\cdot) := \max_{\substack{k\in\{0,\dots,n\}\\k/n\in \cdot, J_{n,k}\ne\emptyset}}\sum_{j\in J_{n,k}}X_j,
\]%\eque
with the convention $\max\emptyset = -\infty$. 
In view of random walks in random sceneries, we think of $\{X_j\}_{j\in\N}$ as rewards, $R_{n,j}$ as the collection of the times (normalized by $1/n$) when the reward $X_j$ is collected, and $m_n$ as the number of copies in the aggregation.
The main result of this paper is the following theorem.
\begin{Thm}\label{thm:1}
Assume that  Assumption  \ref{assump:Rn} and \eqref{eq:RV} hold, that   $m_n$ and $a_n$ satisfy
\equh\label{eq:a_n}
m_n\to\infty \qmand \limn m_n\proba(X_1>a_n) = 1,
\eque
and in addition that
\equh\label{eq:kappa}
m_n \le Cn^\kappa \mbox{ for some }\kappa\in\pp{0,\frac{c_0}{1-1/\alpha}} \mbox{ if } \alpha\ge 1.
\eque 
Then, 
we have
\[%\equh\label{eq:calM}
\frac {M_n}{a_n}\weakto \CRSMa
\]%\eque
in $\SM(E,\wb\R)$, where $\calM_{\alpha,\calR}^{\rm Ca}$ is as in \eqref{eq:RSM2}.
\end{Thm}
\begin{Rem}\label{rem:assump}
We collect some comments on the assumptions.

\begin{enumerate}[(i)]
\item By independence and \eqref{eq:intersection_convergence}, it was shown in 
\citep[Theorem 2.1]{samorodnitsky19extremal} that the {\em joint convergence} of intersections
\equh\label{eq:joint}
\ccbb{R_{n,J}}_{J\subset[\ell]}\weakto \ccbb{\calR_J}_{J\subset[\ell]}
\eque
follows, which is a seemingly much stronger statement than \eqref{eq:intersection_convergence} (which states so only for every fixed $J$).

\item\label{item:2} It is implicitly assumed that for $K_0 = \ceil{1/c_0}$,  
$\proba\pp{\calR_{[K_0]}=\emptyset}=1$. 
In particular, $\calR$ is light (for all $t\in E$, $\proba(t\in \calR) = 0$). Indeed, for $\ell>1/c_0$, \eqref{eq:lattice} and \eqref{eq:p_n} imply that $\proba(R_{n,[\ell]}\cap G \neq \emptyset)\le \sum_{k/n\in G}\proba(k\in R_{n,1})^\ell  \le C(n+1)n^{-\ell c_0}\to 0$, for any $G\in\calG_0$, which combined with \eqref{eq:intersection_convergence} then implies that $\calR_{[\ell]}\cap G = \emptyset$ almost surely. Then take a countable union of sets from $\calG_0$ to cover $E$. 

\item 
The exclusion of a finite number of points of $[0,1]$ is due to the fact that, 
in some examples, $\calR$ is light over only a strict subset $E$ of $[0,1]$ but $[0,1]\setminus E\in\calR$ with probability one. Suppose $0\in[0,1]\setminus E$, as in two of our examples later. 
Then, natural discrete approximations $R_n$ of $\calR$ may satisfy that $0\in R_n$ with probability one. In this way, $\summ j1{m_n}X_j\inddd{k\in R_{n,j}}$ has different orders with $k=0$ and other values, and should be treated differently.  The analysis with $k=0$ is about i.i.d.~random variables and hence standard: however it is not easy to include in our framework the analysis for $k=0$ and other values in a simple unified manner. Therefore 0 is simply excluded from $E$.

\item The lattice condition \eqref{eq:lattice} is, in a sense, necessary for the intersection convergence \eqref{eq:intersection_convergence} from the modeling point of view. On the other hand,  one might view this as that the model is rigid: perturbing $R_n$ a little bit may maintain that $R_n\weakto \calR$ but violate \eqref{eq:intersection_convergence}. 
For an example that $R_n\weakto \calR$ but $R_{n,1}\cap R_{n,2}\not\weakto\calR_1\cap\calR_2$, consider $R_n:=(\vv\tau\cap[0,n])/n$, where $\vv\tau$ is the collection of consecutive renewal times of a renewal process of which the inter-arrival renewal times have $\beta$-regularly-varying tails, as in Section \ref{sec:srs}, but assume instead that the inter-arrival renewal times are continuous random variables.
Then $R_n\weakto\calR$ shall remain the same; however, now $\proba(R_{n,1}\cap R_{n,2} = \{0\}) = 1$, but with $\beta>1/2$, $\calR_1\cap\calR_2$ is almost surely a stable regenerative set. 
\item For $\alpha\ge 1$, the assumption \eqref{eq:kappa} can be relaxed by introducing a deterministic drift term in $M_n$, but it cannot be completely removed for $\alpha\ge 2$. See Remark \ref{rem:drift} for more details.
\end{enumerate}
\end{Rem}

\subsection{Proof of Theorem \ref{thm:1}}%\label{sec:proofs}

For each $n\in\N$, let $(X_{1:m_n},\dots,X_{m_n:m_n})$ be a reordering of $(X_1,\dots,X_{m_n})$ such that
\[
|X_{1:m_n}|\ge \cdots\ge |X_{m_n:m_n}|,
\]
and set $\sigma_n:\{1,\dots,m_n\}\to \{1,\dots,m_n\}$ be such that $X_{\sigma_n(j)} \equiv X_{j:m_n}$. 
Introduce also
\[
\what J_{n,k}:=\ccbb{j\in[m_n]: k/n\in \what R_{n,j}}, k=0,\dots, n, \qmwith \what R_{n,j} := R_{n,\sigma_n(j)}.
\]
In words, $\what R_{n,j}$ is the random closed set corresponding to the $j$-th order statistic among $X_1,\dots,X_{m_n}$, and $\what J_{n,k}$ is the collection of the {\em rankings} of order statistics (instead of the original labels of unordered $X_1,\dots,X_{m_n}$) of which the associated random closed sets cover $k/n$. 
Then we can express
\equh\label{eq:M2}
M_n(\cdot) = \max_{\substack{k\in\{0,\dots,n\}\\k/n\in\cdot, \what J_{n,k}\ne\emptyset}} \sum_{j\in \what J_{n,k}}X_{j:m_n},
\eque
with the convention $\max\emptyset = -\infty$ as before.

Introduce  $\delta_n:=n^{-\gamma}$ for a parameter $\gamma>0$ to be specified later. 
Instead of working with $M_n$ in \eqref{eq:M2} we consider for  $\ell\in\N$ fixed, $n$ large enough so that $m_n\ge \ell$,
\begin{align*}
Y_{n,\ell}(k)&:={\sum_{j\in \what J_{n,k}\cap[\ell]}X_{j:m_n}},\\%\inddd{\what J_{n,k}\ne\emptyset}^*,\\
Z_{n,\ell}(k)&:={\sum_{j\in \what J_{n,k}\setminus[\ell]}X_{j:m_n}\inddd{|X_{j:m_n}|/a_n>\delta_n}},\\
%\inddd{\what J_{n,k}\ne\emptyset}^*\\
W_{n,\ell}(k)&:= {\sum_{j\in \what J_{n,k}\setminus[\ell]}X_{j:m_n}\inddd{|X_{j:m_n}|/a_n\le \delta_n}}, %-b_n(k),%\inddd{\what J_{n,k}\ne\emptyset}^*.
\end{align*}
with the convention $\sum_\emptyset = 0$.
We refer to the processes $Y_{n,\ell}, Z_{n,\ell}, W_{n,\ell}$ as the top, middle and bottom parts, respectively.
Eventually 
only the top part contributes and leads to the desired limit by taking $\ell\to\infty$ (Proposition \ref{prop:top} below).
The middle and bottom parts can be uniformly controlled if $\alpha\in(0,1)$, and the hard work is to deal with the case $\alpha\ge1$. 
The middle part is negligible because with high probability at most a finite fixed number of random closed sets from the middle part will intersect at any time $k=0,\dots,n$ (Lemma \ref{lem:middle}), and the bottom part, centered, is essentially sub-Gaussian and controlled by Bernstein inequality (Lemma \ref{lem:bottom}).

We start with analyzing the top part.
Introduce the corresponding random sup-measure
\equh\label{eq:M_n,ell}
M_{n,\ell}(\cdot)  :=\max_{\substack{k\in\{0,\dots,n\}\\k/n\in\cdot,  
\what J_{n,k}\cap[\ell]\ne\emptyset}}Y_{n,\ell}(k), \quad\ell\in\N.
\eque
Recall $\calM_\ell\topp {\mathsf p}$ and $\calM_{J}\topp{\mathsf p}$ in \eqref{eq:M_l,epsi} here for the convenience,
\[
\calM_\ell\topp {\mathsf p}(\cdot):= \max_{J\subset[\ell]}\calM_{J}\topp{\mathsf p}(\cdot)\qmwith \calM_{J}\topp{\mathsf p}(\cdot):= 
\begin{cases}
\displaystyle\sum_{j\in J}\frac{\varepsilon_j}{\Gamma_j^{1/\alpha}} & \mbox{ if } \calR_J\cap\cdot\ne\emptyset,\\
-\infty & \mbox{ otherwise.}
\end{cases}
%J\subset\N, |J|<\infty.
\]
\begin{Prop}\label{prop:top}

 Under the assumptions in Theorem \ref{thm:1},
 we have  for every  $\ell\in\N$, 
\[%\equh\label{eq:top1}
\frac1{a_n}M_{n,\ell}\weakto \mathsf p^{-1/\alpha}\calM_\ell\topp {\mathsf p}
\]%\eque
in $\SM(E,\wb\R)$ as $n\to\infty$. %, with $\calM_\ell\topp {\mathsf p}$ defined in \eqref{eq:M_l,epsi}.
\end{Prop}
Instead of working with $M_{n,\ell}$ in \eqref{eq:M_n,ell}, we consider the following approximation. Set 
\[%\equh\label{eq:MJ}
\wt M_{n,J}(\cdot):=
\begin{cases}
\sum_{j\in J} X_{j:m_n} & \mbox{ if } \what R_{n,J}\cap \cdot\ne \emptyset,\\
-\infty & \mbox{ otherwise,}
\end{cases} 
\]
with $\what{R}_{n,J}:=\bigcap_{j\in J}\what R_{n,j}$, where $J\subset\N, |J|<\infty$,
%\eque
and
\[
\wt M_{n,\ell}(\cdot):=\max_{J\subset[\ell]}\wt M_{n,J}(\cdot), \quad\ell\in\N.
\]
Introduce $\wb a_n$ be such that $\limn m_n\proba(|X_1|>\wb a_n) = 1$. Note that $\wb a_n\sim \mathsf p^{-1/\alpha}a_n$. 

\begin{Lem}\label{lem:top_PP}
Under the assumptions of Theorem \ref{thm:1}, for $\ell\in\N$, 
\equh\label{eq:top_PP}
\sum_{J\subset[\ell]}\ddelta{\sum_{j\in J}X_{j:m_n}/\wb a_n,\what R_{n,J}}\weakto \sum_{J\subset[\ell]}\ddelta{\sum_{j\in J}\varepsilon_j\Gamma_j^{-1/\alpha},\calR_J}
\eque
in $\mathfrak M_p(\wb \R\times\calF(E))$,  and 
\[%\begin{equation}\label{Step 2}
\frac1{a_n}\wt M_{n,\ell}\weakto \mathsf p^{-1/\alpha}\calM_\ell\topp {\mathsf p}
\]%\end{equation}
in $\SM(E,\wb\R)$, as $n\to\infty$.
\end{Lem}
\begin{proof}[Proof of Lemma \ref{lem:top_PP}]
 For the first part, 
it suffices to show the convergence of the rewards and the reward times, respectively. 
The joint convergence of the rewards  $\{\wb a_n\inv\sum_{j\in J}X_{j:m_n}\}_{J\subset[\ell]}\weakto \{\sum_{j\in J}\varepsilon_j\Gamma_j^{-1/\alpha}\}_{J\subset[\ell]}$ is well-known to follow from \eqref{eq:RV} (e.g.~\citep{resnick87extreme}). The joint convergence of intersections of random closed sets have been recalled in \eqref{eq:joint}. 

For the second part, we first recall that the mapping $\R\times\calF(E)\ni (x,F)\mapsto x\inddds{F\cap\cdot\ne\emptyset}\in\SM(E,\wb\R)$ is continuous (recalling \eqref{eq:indicator*}), and this can be readily checked by definition. Also we recall that for $d\in\N$, the mapping $\SM(E,\wb \R)^d\ni (m_1,\dots,m_d)\mapsto \max_{i=1,\dots,d}m_i\in \SM(E,\wb \R)$ is continuous ($(\max_{i=1,\dots,d}m_i)(\cdot) := \max_{i=1,\dots,d}m_i(\cdot)$) \citep[Theorem 14.6]{vervaat97random}. 
Then, \eqref{eq:top_PP} implies immediately that
\[%\equh\label{eq:joint_J}
\ccbb{\frac1{\wb a_n}\wt M_{n,J}}_{J\subset[\ell]}\weakto \ccbb{\calM_{J}\topp{\mathsf p}}_{J\subset[\ell]}
\]%\eque
in $\SM(E,\wb\R)^{2^\ell}$, by continuous mapping theorem.
\end{proof}
However, it should be clear that $\wt M_{n,\ell}$ and $M_{n,\ell}$ are not identical in general, as explained in the following remark. 
\begin{Rem}%\label{rem:figure}
When the rewards are all non-negative, $M_{n,\ell}$ and $\wt M_{n,\ell}$ are identical. In the presence of some negative values they may differ over some interval $G$ (and we shall show that this occurs with negligible probability as $n\to\infty$ below). Figure \ref{fig:1} provides an illustration, where for comparison we also include 
\[
\what M_{n,\ell}(\cdot):=\max_{\substack{k\in\{0,\dots,n\}\\
k/n\in\cdot, J_{n,k}\ne\emptyset}}\max_{j=1,\dots,m_n}X_j,
\] the limit of which is the CRSM (without aggregations).  
All random sup-measures are coupled, based on the same $\{X_j, R_{n,j}\}_{j=1,\dots,m_n}$. 
In words, the aggregations may push CRSM (corresponding to $\what M_{n,\ell}$) upwards where intersections occur with positive rewards (but never pull CRSM downwards anywhere), while the empirical random sup-measure of our models ($M_{n,\ell}$) is obtained by pushing upwards or pulling downwards $\what M_{n,\ell}$ where intersections occur (the direction depending on the signs of the cumulative rewards at the intersections). 
\end{Rem}
\begin{figure}[ht!]
\begin{center}
\includegraphics[width = \textwidth]{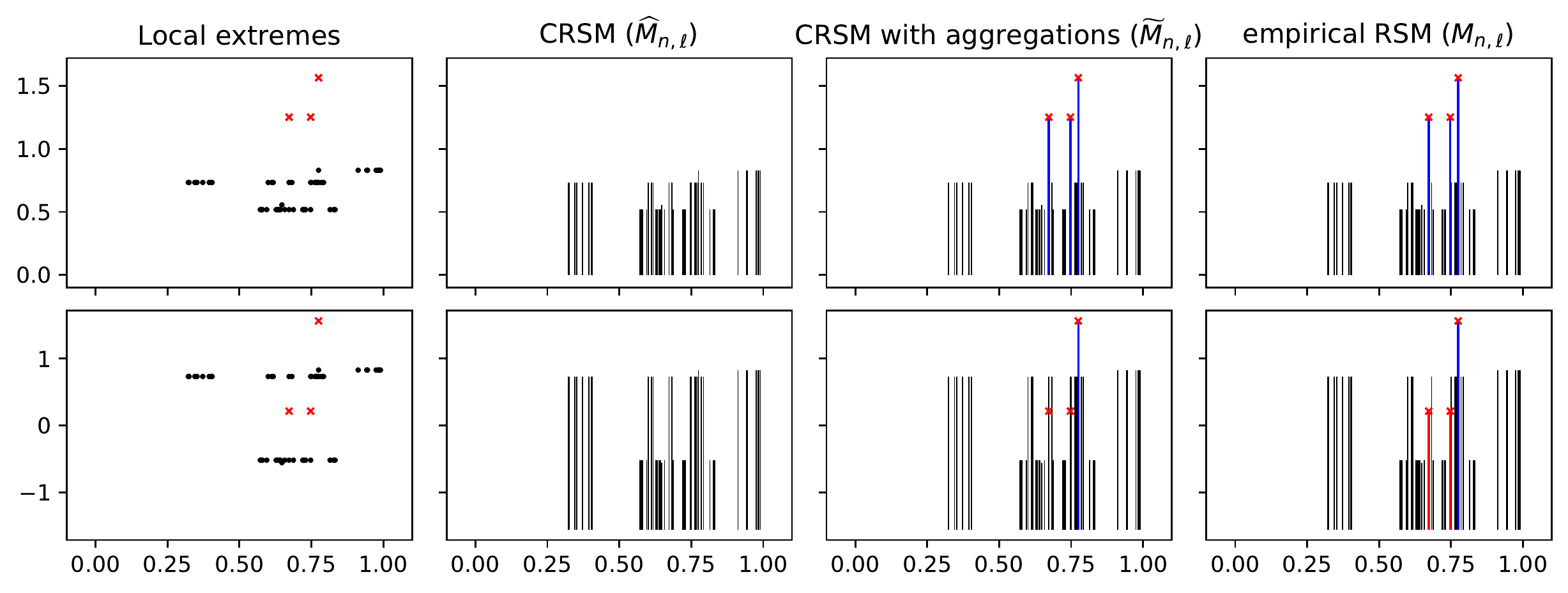}
\end{center}
\caption{An illustration for CRSM ($\what M_{n,\ell}$), CRSM with aggregations ($\wt M_{n,\ell}$) and the empirical random sup-measure of our model ($M_{n,\ell}$), regarding the effect of positive/negative rewards, with $m_n = \ell = 5$.
Top: all rewards are positive. Bottom: some rewards are negative. 
In each plot, the hypograph of the corresponding random sup-measure is plotted. The `push-ups' and `pull-downs' of the CRSM due to intersections are marked in blue and red colors, respectively. 
}\label{fig:1}
\end{figure}
The key to show that $M_{n,\ell}$ and $\wt M_{n,\ell}$ are close  is the following. It is remarkable that it requires very mild assumptions on the random closed sets. 
\begin{Lem}\label{lem:RnJ*}
Under the assumptions \eqref{eq:lattice} and  \eqref{eq:p_n}, for every $\ell$ fixed and open set $G$,
\begin{multline*}
\limn\proba\pp{R_{n,J,\ell}^*\cap G \ne \emptyset \mbox{ or } R_{n,J}\cap G = \emptyset} = 1 \\
\mwith R_{n,J,\ell}^* :=R_{n,J}\setminus\pp{\bigcup_{j\in[\ell]\setminus J}R_{n,j}}, \mfa J\subset[\ell].
\end{multline*}
\end{Lem}
\begin{proof}
In words, the lemma says that with probability going to one, either $R_{n,J}$ does not intersect $G$, or it intersects with $G$  at some location(s) uncovered by any other $R_{n,j}$ for $j\in [\ell]\setminus J$. 
%\begin{proof}
Write the probability of interest as 
\[
\proba(R_{n,J}\cap G = \emptyset) + \proba(R_{n,J,\ell}^*\cap G\ne\emptyset\mid R_{n,J}\cap G \ne\emptyset)\proba(R_{n,J}\cap G\ne\emptyset). 
\]
Let $\what g_{n,J} := \min\{R_{n,J}\cap G\}$ denote the left-most point of the intersection provided it is not empty. Then, 
\begin{align*}
\proba\pp{R_{n,J,\ell}^*\cap G\ne\emptyset\mid R_{n,J}\cap G\ne\emptyset} & \ge \proba\pp{\what g_{n,J}\in R_{n,J,\ell}^*\cap G \mid R_{n,J}\cap G\ne\emptyset} 
\\
& \ge (1-p_n(G))^{\ell-|J|}.%\ge (1-Cn^{-c_0})^{\ell-|J|}. 
\end{align*}
This completes the proof.
\end{proof}

\begin{proof}[Proof of Proposition \ref{prop:top}]

The goal is to show
\[%\begin{equation}\label{eq:top_approx}
\limn\proba\pp{M_{n,\ell}(G) = \wt M_{n,\ell}(G)} = 1 \mfa G\in\calG_0.
\]%\end{equation}
%\begin{proof}[Proof of Lemma \ref{lem:wtM}]
To do so, note that on the event $M_{n,\ell}(G)=-\infty$, we have $J_{n,k}=\emptyset$ for all $k$ such that $k/n\in G$, and hence $\what R_{n,J}\cap G = \emptyset$ for all $J\subset[\ell]$. So $M_{n,\ell}(G) = \wt M_{n,\ell}(G)$ as both sides are $-\infty$. From now on we restrict to the event  $M_{n,\ell}(G)\ne -\infty$. 
We first show that $M_{n,\ell}(G)\le \wt M_{n,\ell}(G)$ always hold {\em without our assumptions} on $R_n,\calR$. 
Suppose that $M_{n,\ell}(G)$ is achieved by the cumulative rewards at certain time $k_0$, namely  
\[
k_0/n\in \what R_{n,\what J_{n,k_0}\cap [\ell]}\cap G
\]
and
\[
M_{n,\ell}(G) = Y_{n,\ell}(k_0) = \sum_{j\in \what J_{n,k_0}\cap [\ell]} X_{j:m_n} = \wt M_{n,\what J_{n,k_0}\cap[\ell]}(\{k_0/n\}).%>-\infty.
\] 
We thus have
 \[
 M_{n,\ell}(G) = \wt M_{n,\what J_{n,k_0}\cap [\ell]}(\{k_0/n\})= \wt M_{n,\what J_{n,k_0}\cap[\ell]}(G)
 \leq \max_{J\subset [\ell]} \wt M_{n,J}(G)=\wt M_{n,\ell}(G).
 \]

Now we show that with probability going to one, $M_{n,\ell}(G)\ge \wt M_{n,\ell}(G)$. 
Assume that $\wt M_{n,\ell}(G)$ is achieved  by $\wt M_{n,J_0}(G)$ for some non-empty set $J_0\subset [\ell]$, that is, $\wt M_{n,\ell}(G) = \wt M_{n,J_0}(G) = \sum_{j\in J_0} X_{j:m_n}$. 
Then one can find 
\[%\equh\label{eq:k1}
k_1/n \in \what R_{n,J_0}\cap G\ne \emptyset \quad\mbox{ such that }\quad  \wt M_{n,J_0}(G) = \wt M_{n,J_0}(\{k_1/n\}).
\]%\eque
It suffices to find such a $k_1$ that in addition satisfies that $J_{n,k_1}\cap[\ell] = J_0$, or equivalently $\what R_{n,J_0,\ell}^*\cap G\ne\emptyset$. This latter event is contained in the event
\[
\bigcap_{\emptyset\subsetneq J\subsetneq [\ell]}\ccbb{R_{n,J,\ell}^*\cap G\ne\emptyset \mbox{ or } R_{n,J}\cap G = \emptyset}.
\]
Lemma \ref{lem:RnJ*} entails that the above event has probability going to one as $n\to\infty$. This completes the proof.
\end{proof}

Next we deal with the middle part. 
\begin{Lem}\label{lem:middle}
With
$\gamma\in(0,{c_0}/\alpha),\delta_n = n^{-\gamma}$, 
we have that under the assumptions \eqref{eq:lattice} and \eqref{eq:p_n},
\[
\lim_{\ell\to\infty}\limsupn\proba\pp{\frac 1{a_n}\max_{k/n\in G}\abs{Z_{n,\ell}(k)}>\epsilon} = 0, \quad \mfa \epsilon>0, G\in\calG_0.
\]
\end{Lem}

\begin{proof}%[Proof of Lemma \ref{lem:middle}]
Throughout we fix $G$.
Consider
\[
\zeta_n(r):= \sum_{\ell+1\le j_1<\cdots<j_r\le m_n}\pp{\prodd s1 r\inddd{|X_{j_s:m_n}/a_n|>\delta_n}}\inddd{\bigcap_{s=1}^r \what R_{n,j_s}\cap G\ne\emptyset}.
\]
Then, the desired result follows from,  for $r>1/(c_0-\alpha\gamma)$, 
\equh\label{eq:Qn}
\limn\proba\pp{\zeta_n(r) > 0} = 0,
\eque
which in words says that the probability that there are at least $r$ different rewards collected by the middle process at certain time $k\in\{0,1,\dots,n\}, k/n\in G$, goes to zero, 
and
\equh\label{eq:zeta_n=0}
\lim_{\ell\to\infty}\limsupn\proba\pp{\frac 1{a_n}\max_{k/n\in G}\abs{Z_{n,\ell}(k)}>\epsilon, \zeta_n(r) = 0} = 0, \mfa \epsilon>0.
\eque
For \eqref{eq:zeta_n=0}, it suffices to notice that on the event $\zeta_n(r)=0$ at every location $k$ there are at most $r-1$ different $j$ among $\ell+1,\dots,m_n$ such that $k\in \what R_{n,j}$. 
Therefore, 
\begin{align*}
\limsupn  &\proba\pp{\frac 1{a_n}\max_{k/n\in G}\abs{Z_{n,\ell}(k)}>\epsilon, \zeta_n(r) = 0}\\
& = \limsupn \proba\pp{\frac 1{a_n}\max_{k/n\in G}\abs{Z_{n,\ell}(k)}>\epsilon\mmid \zeta_n(r) = 0}\proba(\zeta_n(r) = 0) \\
&\le \proba\pp{\mathsf p^{-1/\alpha}(r-1)\Gamma_{\ell+1}^{-1/\alpha}>\epsilon},
\end{align*}
whence \eqref{eq:zeta_n=0} follows.
Now it remains to show \eqref{eq:Qn}.
Introduce
\[
\proba\pp{\bigcap_{j=\ell+1}^{\ell+r} \what R_{n,j}\cap G\ne\emptyset}  = \proba\pp{R_{n,[r]}\cap G\ne\emptyset} =: \rho_n(r).
\]Note that we always have
\[
\rho_n(r) \le \sum_{k/n\in G} \proba\pp{k/n\in R_{n,[r]}} \le (n+1)\max_{k/n\in G}p_n(k)^r = (n+1)p_n(G)^r \le Cn^{1-r c_0},
\]
the last step follows from the assumption \eqref{eq:p_n}.
Moreover, 
\begin{align}
\proba(\zeta_n(r)>0)&\le \esp \zeta_n(r)
\le\esp\pp{\sum_{1\le j_1<\cdots<j_r\le m_n}\prodd s1r\inddd{|X_{j_s}|>a_n\delta_n}\inddd{\bigcap_{s=1}^r\what R_{j_s}\cap G\ne\emptyset}}\nonumber
\\
 &\le \binom {m_n}{r}\cdot\proba\pp{|X_1|>a_n\delta_n}^r \cdot \rho_n(r).\label{eq:an_deltan}
 \end{align}
Here and in a few places below, we shall discuss two situations $a_n\delta_n\to\infty$ and $\limsupn a_n\delta_n<\infty$, respectively. Note that these two cases do not include the situation that $\limsupn a_n\delta_n = \infty$ but $a_n\delta_n\not\to\infty$; but the same conclusion for this third case can be derived from the first two by considering subsequences. 
 
 If $\limsupn a_n\delta_n<\infty$, then $a_n$ is bounded by $C\delta_n^{-1}$, and \eqref{eq:an_deltan} is bounded by
$C\delta_n^{-\alpha r}\rho_n(r) = Cn^{r\alpha\gamma+1-c_0r}$ up to a slowly-varying function in $n$, and the assumption $r>1/(c_0-\alpha\gamma)$ guarantees that it goes to zero. 
If $a_n\delta_n\to\infty$, then \eqref{eq:an_deltan} is bounded from above by
\begin{align*}
 \frac C{r!}m_n^r (a_n\delta_n)^{-\alpha r}L^r(a_n\delta_n)\rho_n(r) &\le C\delta_n^{-\alpha r}\rho_n(r) \cdot \frac{m_n^r}{a_n^{\alpha r}} L^r(a_n)\frac{L^r(a_n\delta_n)}{L^r(a_n)}\\
 &\le C\delta_n^{-\alpha r}n^{\epsilon_1}\rho_n(r),
\end{align*}
for $\epsilon_1>0$ arbitrarily small by Potter's bound (recall also \eqref{eq:a_n}). So the above is bounded by  $Cn^{r(\alpha\gamma-c_0)+1+\epsilon_1}$. To complete the proof of \eqref{eq:Qn}, it suffices to take $\epsilon_1>0$ such that $r(\alpha\gamma-c_0)+1+\epsilon_1<0$.
\end{proof}
%\NP
Next we deal with the bottom part.
\begin{Lem}\label{lem:bottom}
For all $\alpha\ge1$, $\gamma\in(0,c_0/\alpha)$ with $\delta_n =n^{-\gamma}$, under assumptions \eqref{eq:lattice} and \eqref{eq:p_n},
for all $\ell\in\N$,
\[
\limsupn\proba\pp{\frac 1{a_n}\max_{k/n\in G}\abs{W_{n,\ell}(k)}>\epsilon} = 0 \mfa \epsilon>0.
\]
\end{Lem}
\begin{proof}
Notice also that for all $\ell'\in\N$ fixed,
$\limsup_{n\to\infty}\proba(|X_{\ell':m_n}|/a_n\le \delta_n)=0$. Therefore, with
\[
V_{n,j}(k):= X_j\inddd{|X_j|/a_n\le\delta_n}\inddd{k/n\in R_{n,j}},
\] 
 it suffices to focus on
the event $\{a_n\inv \max_{k/n\in G}|\summ j1{m_n}V_{n,j}(k)|>\epsilon\}$. 
Introduce
\equh\label{eq:bnk}
\wt W_n(k):=\sum_{j=1}^{m_n}(V_{n,j}(k)-\esp V_{n,j}(k))\qmand b_n(k) :=  m_np_n(k)\esp\pp{X_1\inddd{|X_1|\le a_n\delta_n}}.
\eque 
We write
%\equh\label{eq:drift}
\[
\summ j1{m_n}V_{n,j}(k) = \wt W_n(k)+b_n(k).
\]%\eque
We shall need below that for every $\epsilon_1>0$ there exists a constant $C$ such that 
 \equh\label{eq:SV}
 \esp \pp{|X_1|^\alpha\inddd{|X_1|\le a_n\delta_n}}\le C a_n^{\epsilon_1}L(a_n), \alpha>0.
 \eque
To see this, write $F(x) = \proba(|X_1|\le x)$ and $\wb F(x) = 1-F(x) = x^{-\alpha}L(x)$ (recall \eqref{eq:RV}). It suffices to consider that $a_n\delta_n\to\infty$. Consider then a slowly varying sequence $\{d_n\}_{n\in\infty}$ with $d_n\to\infty$ and $d_n= o(a_n\delta_n)$,  
\begin{align*}
\esp \pp{|X_1|^\alpha\inddd{|X_1|\le a_n\delta_n}} & = \int_0^{a_n\delta_n}x^\alpha dF(x) \le d_n^\alpha + \int_{d_n}^{a_n\delta_n}x^\alpha dF(x) \\
&\le 2d_n^\alpha+\alpha\int_{d_n}^{a_n\delta_n}x^{\alpha-1}\wb F(x)dx \\
%& = 2d_n^\alpha+\alpha\int_{d_n}^{a_n\delta_n} x\inv L(x)dx 
& = 2d_n^\alpha+\alpha L(a_n)\int_{d_n}^{a_n\delta_n}\frac{x^{-1}L(x)}{L(a_n)}dx \le  2d_n^\alpha+Ca_n^{\epsilon_1}L(a_n),
\end{align*}
where the second inequality is from integration by part, and 
the last step by Potter's bound. Assuming in addition that $d_n^\alpha = o(a_n^{\epsilon_1}L(a_n))$ we have \eqref{eq:SV}. 

We first show that the drift term is negligible:
\equh\label{eq:drift}
\limn \frac1{a_n}\max_{k/n\in G}|b_n(k)|  = 0.
\eque
In the case $\alpha=1$, \eqref{eq:drift} is obvious if $\limsupn a_n\delta_n<\infty$.  So assume that $a_n\delta_n\to\infty$. 
By \eqref{eq:SV}, we have
\[
\frac1{a_n}\max_{k/n\in G}{|b_n(k)|}  \le \frac{m_n}{a_n}p_n(G)\esp\pp{|X_1|\inddd{|X_1|\le a_n\delta_n}}\le Ca_n^{\epsilon_1}p_n(G),
\]
which goes to zero if $m_n$ (and hence $a_n$) grows at any rate that is polynomial in $n$ by taking $\epsilon_1>0$ small enough. 
In the case $\alpha>1$, %$w_n  = -m_np_n(k)\esp \pp{X_1\inddd{|X_1|>a_n\delta_n}}$, so
\[%\equh\label{eq:drift>1}
\frac1{a_n}\max_{k/n\in G}\abs{b_n(k)} \le \frac{m_np_n(G)}{a_n}\cdot\esp \pp{|X_1|},\]%\eque
so under the assumption that $m_n\le C n^\kappa$ for all $\kappa<c_0/(1-1/\alpha)$, \eqref{eq:drift} holds too. 

It remains to  prove that 
for all $\epsilon>0$,
\equh\label{eq:ineq_W}
\limsupn\proba\pp{\frac 1{a_n}\max_{k/n\in G}\abs{\wt W_{n}(k)}>\epsilon} =0.
\eque
We shall first apply the union bound on the maximal probability of interest above
\[
\proba\pp{\frac 1{a_n}\max_{k/n\in G}\abs{\wt W_{n}(k)}>\epsilon}\le (n+1)\max_{k/n\in G}\proba\pp{\frac1{a_n}\abs{\wt W_n(k)}>\epsilon},
\]
and then  establish below an exponential bound for the probability on the right-hand side. 
Note that for every $n,k\in\N$, $\{V_{n,j}(k)\}_{j\in\N}$ are i.i.d.~ and $|V_{n,j}(k)|\le a_n\delta_n$. 
Introduce $w_{n,\alpha}(k) = m_n \esp V_{n,1}^2(k)$. Then, Bernstein's inequality \citep[Section 2.7]{boucheron13concentration} tells that
\equh\label{eq:Bernstein}
\proba\pp{\frac{1}{a_n}\abs{\wt W_{n}(k)}\geq \epsilon} 
\le 2\exp\pp{-\frac{(a_n\epsilon)^2/2}{w_{n,\alpha}(k)+a_n\delta_n\cdot a_n\epsilon/3}}.\eque

Assume that $a_n\delta_n\to\infty$ for now. 
If $\alpha<2$, then for $k/n\in G$,
\begin{align*}
w_{n,\alpha}(k) & = m_n\esp V_{n,1}^2(k) \le m_n\esp \pp{X_j^2\inddd{|X_j|\le a_n \delta_n} }\cdot p_n(G)
\\
&\leq C m_n(a_n\delta_n)^{2-\alpha}L(a_n\delta_n) p_n(G)\\
& \le Ca_n^2\delta_n^{2-\alpha}\frac{L(a_n\delta_n)}{L(a_n)}p_n(G) \le Ca_n^2\delta_n^{2-\alpha-\epsilon_1}p_n(G),
\end{align*}
for some $\epsilon_1\in(0,2-\alpha)$,
where the second inequality follows from Karamata's theorem, the third from the assumption \eqref{eq:RV}, and the last from Potter's bound. 
So
\[
\proba\pp{\frac{1}{a_n}\abs{\wt W_{n}(k)}\geq \epsilon}
\le2\exp\pp{-\frac{C}{ p_n(G)+\delta_n}}, k/n\in G.
\]
Therefore \eqref{eq:ineq_W} holds.

If $\alpha>2$, $\esp V^2_{n,1}(k) \le Cp_n(G)$, and hence $w_{n,\alpha}(k)\le Cm_np_n(G)$. Thus \eqref{eq:Bernstein} says that
\equh\label{eq:alpha>2}
\proba\pp{\frac1{a_n}\abs{\wt W_n(k)}>\epsilon} \le 2\exp\pp{-\frac{C}{m_na_n^{-2}p_n(G)+\delta_n}}, k/n\in G.
\eque
Observe that $m_n a_n^{-2}p_n(G)\le Ca_n^{\alpha-2}L\inv (a_n)p_n(G)$, and that $m_n\in RV_\kappa$ implies $a_n\in RV_{\kappa/\alpha}$. Therefore, $m_n\le Cn^\kappa$ with $\kappa<c_0/(1-2/\alpha)$ implies \eqref{eq:ineq_W}. 

 For $\alpha=2$, if $\esp X_1^2<\infty$, then \eqref{eq:alpha>2} remains to hold, and $m_n$ can grow at any polynomial rate by the same argument above. So assume $\esp X_1^2=\infty$. By \eqref{eq:SV}, we obtain a similar inequality as \eqref{eq:alpha>2}, with $m_na_n^{-2}p_n(G)$ replaced by $m_na_n^{-2+\epsilon_1}L(a_n)p_n(G)\le C a_n^{\epsilon_1}p_n(G)$ (recall \eqref{eq:a_n}). Then \eqref{eq:ineq_W} still holds as long as $m_n$ (and hence $a_n$) grows no faster than any polynomial rate, as $\epsilon_1>0$ can be arbitrarily small.
 
 We have proved \eqref{eq:ineq_W} under the assumption that $a_n\delta_n\to\infty$. 
 In the case $\limsupn a_n\delta_n<\infty$, one can use the upper bounds for $w_{n,\alpha}(k)/a_n^2$ as before. This completes the proof.
\end{proof}
%Now we prove the main theorem.
\begin{proof}[Proof of Theorem \ref{thm:1}]
Recall the convergence-determining class $\calG_0$ in \eqref{eq:G_0}. It suffices to show that, for all $d\in\N$ and all disjoint open intervals $G_{i}\in\calG_0$, $i=1,\dots,d$,
\[
\pp{\frac{1}{a_n}M_n(G_1),\cdots, \frac{1}{a_n}M_n(G_d)}\weakto
%\pp{\mathsf p^{1/\alpha}\calM(G_1), \cdots, \mathsf p^{1/\alpha}\calM(G_d)}
\pp{\CRSMa(G_1), \cdots,\CRSMa(G_d)}
\]
as $n\to\infty$.
By Proposition \ref{prop:RSM3}, we have 
$\calM_\ell\topp {\mathsf p}\weakto \mathsf p^{1/\alpha}\CRSMa$. For the sake of simplicity, assume that $\proba(\calR\cap G_i\ne\emptyset)>0$ for all $i=1,\dots,d$, whence $\lim_{\ell\to\infty}\limsupn\proba(M_{n,\ell}(G_i)=-\infty \mbox{ or } M_n(G_i) = -\infty) = 0$. 
By Proposition \ref{prop:top}, we have for all $\ell\in\N$,
\[
\pp{\frac{1}{a_n}M_{n,\ell}(G_1),\cdots, \frac{1}{a_n}M_{n,\ell}(G_d)}\weakto
\pp{\mathsf p^{-1/\alpha}\calM_\ell\topp {\mathsf p}(G_1), \cdots, \mathsf p^{-1/\alpha}\calM_\ell\topp {\mathsf p}(G_d)}.
\]
Then, by \citep[Theorem 3.2]{billingsley99convergence}, it remains to check that for any $\epsilon>0$,
\begin{equation}\label{eq:approx}
\lim_{\ell\to\infty}\limsupn \proba\pp{\frac{1}{a_n}|M_n(G)-M_{n,\ell}(G)|>\epsilon, M_{n,\ell}(G)\ne-\infty,M_n(G)\ne-\infty}=0.
\end{equation}
For $\alpha< 1$, for any $x>0$, observe that (i)
restricted to the event $M_{n,\ell}(G)\ne-\infty$ and $M_n(G)\ne-\infty$,
$|M_n(G)-M_{n,\ell}(G)|\le \summ j{\ell+1}{m_n}|X_{j:m_n}|$, (ii) for every $x>0$, 
\[
\lim_{\ell\to\infty}\proba\pp{\sum_{j=\ell+1}^{m_n}|X_{j:m_n}|\le \summ j1{m_n}|X_j|\inddd{|X_j|\le a_nx}} = 1,
\]
and (iii) for every $\epsilon$ fixed, there exists a constant $C$ such that for $n$ large enough (depending on $x$), 
\begin{align*}
\limsupn\proba\pp{\frac1{a_n}\esp\abs{\summ j1{m_n}|X_j|\inddd{|X_j|\le a_nx}}>\epsilon}& \le \frac1\epsilon\limsupn\frac{m_n}{a_n}\esp\pp{|X_1|\inddd{|X_1|\le a_nx}}\\
&=\frac C\epsilon\limsupn \frac{m_n}{a_n}(a_n x)^{1-\alpha}L(a_nx) \\
&= \frac{C}\epsilon x^{1-\alpha}.
\end{align*}
Taking $x\downarrow 0$ and combining the three facts above complete the proof of \eqref{eq:approx} for $\alpha<1$. 
For $\alpha\ge 1$, 
restricted to the event $M_{n,\ell}(G)\ne-\infty$ and $M_n(G)\ne-\infty$
we write
\[
\abs{M_{n}(G)-M_{n,\ell}(G)} \le\max_{k/n\in G} \pp {\abs{Z_{n,\ell}(k)}+\abs{W_{n,\ell}(k)}}.
\]
The desired \eqref{eq:approx} follows from Lemma \ref{lem:middle} and Lemma \ref{lem:bottom}. 
\end{proof}

\begin{Rem}\label{rem:drift}
We did not search for the most general assumption. For example, from the proof it is clear that in the case $\alpha\ge 1$, we could introduce a drift in the empirical random sup-measure
\[
M_{n}(\cdot) := \max_{\substack{k\in\{0,\dots,n\}\\k/n\in \cdot, J_{n,k}\ne\emptyset}}\pp{\sum_{j\in J_{n,k}}X_j-b_n(k)},% \qmwith b_n(k) := m_n\esp \pp{X_1\inddd{|X_1|\le a_n\delta_n}\inddd{k\in R_{n,1}}}.
\]
with $b_n(k)$ as in \eqref{eq:bnk}.
In this way, for the bottom part it suffices to consider now
\[
W_{n,\ell}(k):= {\sum_{j\in \what J_{n,k}\setminus[\ell]}X_{j:m_n}\inddd{|X_{j:m_n}|/a_n\le \delta_n}}-b_n(k).
\]
The same analysis goes through for the rest and there is no need to guarantee \eqref{eq:drift}, and hence the assumption \eqref{eq:kappa} can be dropped. As a consequence, the same convergence holds with $m_n$ grows at any rate for $\alpha<2$, and any polynomial rate $n^\kappa$ with $\kappa\in (0,c_0/(1-2/\alpha))$ for $\alpha\ge 2$ (the only place this is needed is when showing \eqref{eq:alpha>2}). 
However, there seems to be no canonical way to pick $b_n(k)$ which depends on $\delta_n$. 

One can also see that $\delta_n = n^{-\gamma}$ is chosen for convenience. One may take $p_n(G)$ and $\delta_n$ to be slowly varying function and establish the same convergence for appropriately chosen $m_n, a_n$ under suitable conditions. The exact conditions would involve slowly-varying functions, and we do not pursue.
\end{Rem}

\NP
\section{Examples}\label{sec:example}
Here we present a few examples for our framework. For each example, we first introduce the random closed set $\calR$ in $\CRSMa$ and then $R_n$ (and the discrete model behind) so that Assumption \ref{assump:Rn} can be verified and Theorem \ref{thm:1} can be applied. %
Throughout, $\{\Gamma_j\}_{j\in\N}$ are consecutive arrival times of a standard Poisson process, independent from all other random variables. The first two examples are CRSMs without aggregations. We represent random sup-measures as in $\SM(E,\wb\R_+)$ for the sake of simplicity (see Remark \ref{rem:RSM}). 
\subsection{Independently scattered random sup-measures}%\label{sec:isRSM}
An independently scattered $\alpha$-Fr\'echet random sup-measures ($\alpha>0$) on $[0,1]$ with Lebesgue control measure takes the form
\[
\calM_\alpha^{\rm is}(\cdot)  \eqd \sup_{j\in\N}\frac1{\Gamma_j^{1/\alpha}}\inddd{\calV_j\in \cdot}.
\]
Here, we take $\{\calV_j\}_{j\in\N}$ as i.i.d.~uniform random variables on $[0,1]$, independent from $\{\Gamma_j\}_{j\in\N}$. 
With probability one the random closed sets do not intersect. 
This is the same as an $(\alpha,\calR)$-CRSM with $\calR\eqd\{\calV_1\}$. 
To put this example into our aggregation framework, we simply consider $V_n$ to be uniformly distributed over $\{0,1,\dots,n\}/n$. Then $R_n := \{V_n\}\weakto\{\calV_1\}$ in $\calF([0,1])$. All the assumptions are trivial to verify. It is quite straightforward to see that this example can be extended to more general control measure $\nu$, instead of Lebesgue, on $[0,1]$ such that, $\nu$ has a density that is continuous  on $[0,1]$ (and hence bounded), with $V_n$ accordingly constructed. We omit the details.

\subsection{Karlin random sup-measures}\label{sec:Karlin}
The  $\alpha$-Fr\'echet Karlin random sup-measure on $[0,1]$ with control measure $\nu$ and parameter $\beta\in(0,1)$, denoted by $\calM_{\alpha,\beta}^{\rm K}$ has several representations \citep{durieu18family}. 
For the sake of simplicity, assume that $\nu$ is a probability measure on $[0,1]$. Introduce 
\[
\calR^{\rm K}_{\beta}\eqd\bigcup_{k=1}^{Q_\beta}\{\calV_k\},
\]
where $\beta\in(0,1)$, 
 $Q_\beta$ is the Sibuya distribution \citep{sibuya79generalized} determined by $\esp z^{Q_\beta} = 1-(1-z)^\beta$ for $z\in[0,1]$, and $\{\calV_{k}\}_{k\in\N}$ be i.i.d.~random element in $[0,1]$ with law $\nu$ independent from $Q_\beta$.
Assuming that $\nu$ as a continuous density function on $[0,1]$, the independent copies do not intersect with probability one. 
Then, the Karlin random sup-measure can be defined as
\[
\calM_{\alpha,\beta}^{\rm K}(\cdot) :=  \calM_{\alpha,\calR^{\rm K}_\beta}^{\rm C}(\cdot) \eqd \sup_{j\in\N}\frac1{\Gamma_j^{1/\alpha}}\inddd{\calR^{\rm K}_{\beta,j}\cap\cdot\ne\emptyset}
\]
with $\{\calR^{\rm K}_{\beta,j}\}_{j\in\N}$ as  i.i.d.~copies of $\calR^{\rm K}_\beta$.
Note that $Q_\beta\weakto 1$ as $\beta\uparrow 1$, so $\calR_1^{\rm K}\eqd \{\calV_1\}$ and $\calM_{\alpha,1}^{\rm K}\eqd \calM_\alpha^{\rm is}$.

Now to put this random sup-measure into our aggregation framework, it suffices to find discrete random variables $V_n\in\{0,\dots,n\}/n$ such that $V_n\weakto \calV_1$, and consider its i.i.d.~copies $\{V_{n,i}\}_{i\in\N}$, independent from a Sibuya random variable $Q_\beta$. Then, $\bigcup_{i=1}^{Q_\beta}\{V_{n,i}\} \weakto \calR^{\rm K}_\beta$ in $\calF([0,1])$.
We verify the assumptions for $\nu= \rm Leb$, and in this case recall that $V_n$ is uniformly distributed over $\{0,\dots,n\}/n$. Then notice
\[
p_n(k) = \proba\pp{\frac kn\in R_n} = \esp\pp{1-\proba\pp{V_{n,1} \ne \frac kn}^{Q_\beta}} = \proba\pp{V_{n,1} = \frac kn}^\beta,
\]
and thus $p_n(G) = (n+1)^{-\beta}$. One can verify similarly that Assumption \ref{assump:Rn} remains true as long as $\nu$ has a continuous and bounded density on $[0,1]$. 
\begin{Rem}
The Karlin random sup-measure was introduced in \citep{durieu18family} for $E = \R_+$ and $\nu$ as the Lebesgue measure. The extension to more general $\nu$ is obvious. However, the model here that leads to the Karlin random sup-measure {\em is much simpler} than the discrete-time model investigated in \citep{karlin67central,durieu18family}. In \citep{durieu18family}, instead of aggregation over a family of $m_n$ i.i.d.~chains of rewards, the model therein can be viewed as an aggregation of a random number, say $K_n$, of reward chains, where the rewards are i.i.d.~with regularly-varying tails, but the reward times $\{R_{n,j}\}_{j=1,\dots,K_n}$ are {\em dependent and mutually exclusive}.  See \citep{durieu18family} for more details.
\end{Rem}

\subsection{Stable-regenerative random sup-measures}\label{sec:srs}
By stable-regenerative random sup-measures we refer to three subclasses of CRSMs with aggregations, where in each case $\calR$ is based on a variation of the stable-regenerative set. For all three CRSMs with aggregations, aggregations occur with probability one for $\beta>1/2$, due to the almost-sure intersection of the underlying random closed sets.
\subsubsection*{Standard stable-regenerative sets (starting from the origin)}
Let $\calR^{\rm sr}_\beta$ be an ordinary $\beta$-stable regenerative set. This can be defined in law as the closure of the image of a $\beta$-stable subordinator \citep{bertoin99subordinators,giacomin07random}. So $\calR^{\rm sr}_\beta$ is a random closed set in $[0,\infty)$ and 0 is a fixed point ($\proba(0\in\calR^{\rm sr}_\beta) = 1$). Let $\calR^{\rm sr}_{\beta}$ and $\{\calR^{\rm sr}_{\beta,j}\}_{j\in\N}$ be i.i.d.~ $\beta$-stable regenerative sets. It is well known \citep{bertoin99subordinators} that with
\equh\label{eq:srs_intersection}
\bigcap_{j=1}^\ell\calR^{\rm sr}_{\beta,j}\eqd  \begin{cases}
\calR^{\rm sr}_{\beta_\ell}  & \mbox{ if } \beta_\ell\in(0,1),\\
\emptyset & \mbox{ otherwise,}
\end{cases}\qmwith \beta_\ell := \ell\beta-\ell+1\in(0,1).
\eque
The random sup-measure in this case is
\[%\equh\label{eq:srRSM1}
\calM^{\rm sr}_{\alpha,\beta}(\cdot) := \calM_{\alpha,\calR^{\rm sr}_\beta}^{\rm Ca} (\cdot)\eqd \sup_{t\in\cdot}\sum_{j=1}^\infty\frac1{\Gamma_j^{1/\alpha}}\inddd{t\in\calR^{\rm sr}_{\beta,j}},
\]%\eque
the right-hand side above uses the representation involving upper-semi-continuous functions (similar to \eqref{eq:ssrRSM} below as first introduced in the literature). 

Now we look at the corresponding discrete model. It is well known that $\calR^{\rm sr}_\beta$ arises as the scaling limit of the set of renewal times of a renewal process with regularly varying inter-arrival renewal times. In particular, consider i.i.d.~$\indn Y$ taking values in $\N$, and 
\equh\label{eq:tau}
\vv\tau\equiv \{\tau_j:j=0,1,\dots\} \qmwith \tau_0:=0, \tau_j:=Y_1+\cdots+Y_j, j\in\N.
\eque
So $\vv\tau$ is the collection of all renewal times, and it is well known that $R_n:=\vv\tau/n\cap[0,1]\weakto \calR^{\rm sr}_\beta\cap[0,1]$ \citep[Appendix A.5]{giacomin07random}.
To put this example into our aggregation framework, we shall consider $E = (0,1]$ and exclude the origin in particular, as obviously $0\in R_n$ and at $k=0$ all the rewards will be collected, implying an asymptotic behavior that is qualitatively different from those at other times $k/n\in\{1,\dots,n\}$. 

Now we verify our assumptions. 
Note that $p_n(k) = \proba(k\in\vv\tau) \equiv u(k)$ is nothing but the renewal mass function of the renewal process, which has been well investigated in the literature. In particular, it follows from \citep{doney97onesided} that with $p_Y(k) \equiv \proba(Y = k)$ under the assumption
\equh\label{eq:Y}
\wb F_Y(n)\equiv \proba(Y>n) \sim n^{-\beta}L(n) \qmand  \sup_{n\in\N}\frac{np_Y(n)}{\wb F_Y(n)}<\infty,
\eque
we have 
\equh\label{eq:u}
u(n) = \proba(n\in \vv\tau) \sim n^{\beta-1}\frac{\Gamma(1-\beta)}{\Gamma(\alpha)L(n)} \mmas n\to\infty.
\eque
Now the assumption \eqref{eq:p_n} follows as $G$ is an interval bounded away from zero. 
It remains to verify the intersection convergence \eqref{eq:intersection_convergence}. 
Indeed, let $\{R_{n,j}\}_{j\in\N}$ be i.i.d.~copies of $R_n$ and $R_{n,[\ell]}=\bigcap_{j=1}^\ell R_{n,j}$ as before. It suffices to remark that, with $\beta_\ell\in(0,1)$, 
the simultaneous renewals of $\ell$ i.i.d.~renewal processes with parameter $\beta$ form a renewal process with parameter $\beta_\ell$, so 
 $R_{n,[\ell]}$ has the same law $(\vv\tau_{\beta_\ell}\cap\{0,\dots,n\})/n$. 
 (For details, see \citep[Appendix A]{samorodnitsky19extremal}.)
 With $\beta_\ell\le0$, $\bigcap_{j=1}^\ell\calR_{\beta,j}^{\rm sr} = \{0\}$ almost surely, and hence $R_{n,[\ell]}\weakto\{0\}$ \citep[Theorem 2.1]{samorodnitsky19extremal}. \subsubsection*{Randomly shifted stable-regenerative sets}
Introduce $\calR_\beta^{\rm srs} :=\calV_\beta+\calR_\beta^{\rm sr}$, where $\calR_\beta^{\rm sr}$ is as above and $\calV_{\beta}$ is a random variable taking values in $[0,1]$ with probability density function $(1-\beta)v^{-\beta}dv$, independent from $\calR_\beta^{\rm sr}$. 
Then, 
\equh\label{eq:ssrRSM}
\calM^{\rm srs}_{\alpha,\beta}(\cdot) :=\calM_{\alpha,\calR^{\rm srs}_\beta}^{\rm Ca}(\cdot) \eqd \sup_{t\in\cdot}\sum_{j=1}^\infty\frac1{\Gamma_j^{1/\alpha}}\inddd{t\in\calR^{\rm srs}_{\beta,j}},
\eque
where $\{\calR_{\beta,j}\}^{\rm srs}$ are i.i.d.~copies of $\calR_\beta^{\rm srs}$, independent from $\{\Gamma_j\}_{j\in\N}$.
This random sup-measure was introduced in \citep{samorodnitsky19extremal}, and therein the dichotomy \eqref{eq:srs_intersection}  with $\calR^{\rm sr}_\beta$ replaced by $\calR^{\rm srs}_\beta$ was proved to remain true. So for $\beta>1/2$ aggregations occur with probability one. 

To put this in our aggregation framework, one might modify the construction of the previous example, by introducing in addition an appropriate discretization of $\calV_\beta$. Alternatively, there is a canonical construction of  $R_n\weakto \calR_\beta^{\rm srs}$ in $\calF([0,1])$, in view of renewal times of certain dynamical systems from infinite ergodic theory. Such a construction has played a crucial role in the studies of a large family of stable processes with long-range dependence in \citep{rosinski96classes,owada15functional,lacaux16time,samorodnitsky19extremal}. We refer to the presentations therein (and Assumption \ref{assump:Rn} was known to hold). 

\begin{Rem}%\label{rem:figure}
Figure \ref{fig:2} provides an illustration comparing $\calM_\alpha^{\rm is}$ (independently scattered), $\CRSM$ and $\CRSMa$, the last two using $\calR=\calR^{\rm srs}_\beta$, where for each random sup-measure, a realization of the underlying point process 
\eqref{eq:PPP_convergence}, \eqref{eq:Choquet}, \eqref{eq:CRSMa_PP} respectively
and the corresponding hypograph of the random sup-measure are provided.
Recall that 
for every sup-measure $m$, there exists a unique upper-semi-continuous function $f$ (known as the sup-derivative of $m$) such that $m(\cdot) = \sup_{t\in\cdot}f(t)$, and the corresponding unique hypograph of $m$ is the closed set ${\rm hypo}(m):=\{(u,x)\in[0,1]\times\wb \R_+:x\le f(u)\}$. 
With the same $\alpha,\calR$, $\CRSM$ and $\CRSMa$  can be coupled 
by using the right-hand sides of \eqref{eq:CRSM} and \eqref{eq:CRSMa} as definitions sharing the same $\{(\Gamma_\ell,\calR_\ell)\}_{\ell\in\N}$. In this way almost surely $\CRSM(A)\le \CRSMa(A)$ for all $A\subset[0,1]$, or equivalently ${\rm hypo}(\CRSM)\subset{\rm hypo}(\CRSMa)$,  as illustrated in Figure \ref{fig:2}. In words, ${\rm hypo}(\CRSMa)$ is obtained by pushing upwards  ${\rm hypo}(\CRSM)$ at those locations where intersections and hence aggregations occur (e.g.~from the value $\max\{\Gamma_i^{-1/\alpha},\Gamma_j^{-1/\alpha}\}$ at $\calR_i\cap\calR_j$ up to $\Gamma_i^{-1/\alpha}+\Gamma_j^{-1/\alpha}$).
\end{Rem}
\begin{figure}[ht!]
\begin{center}
\includegraphics[width = 0.75\textwidth]{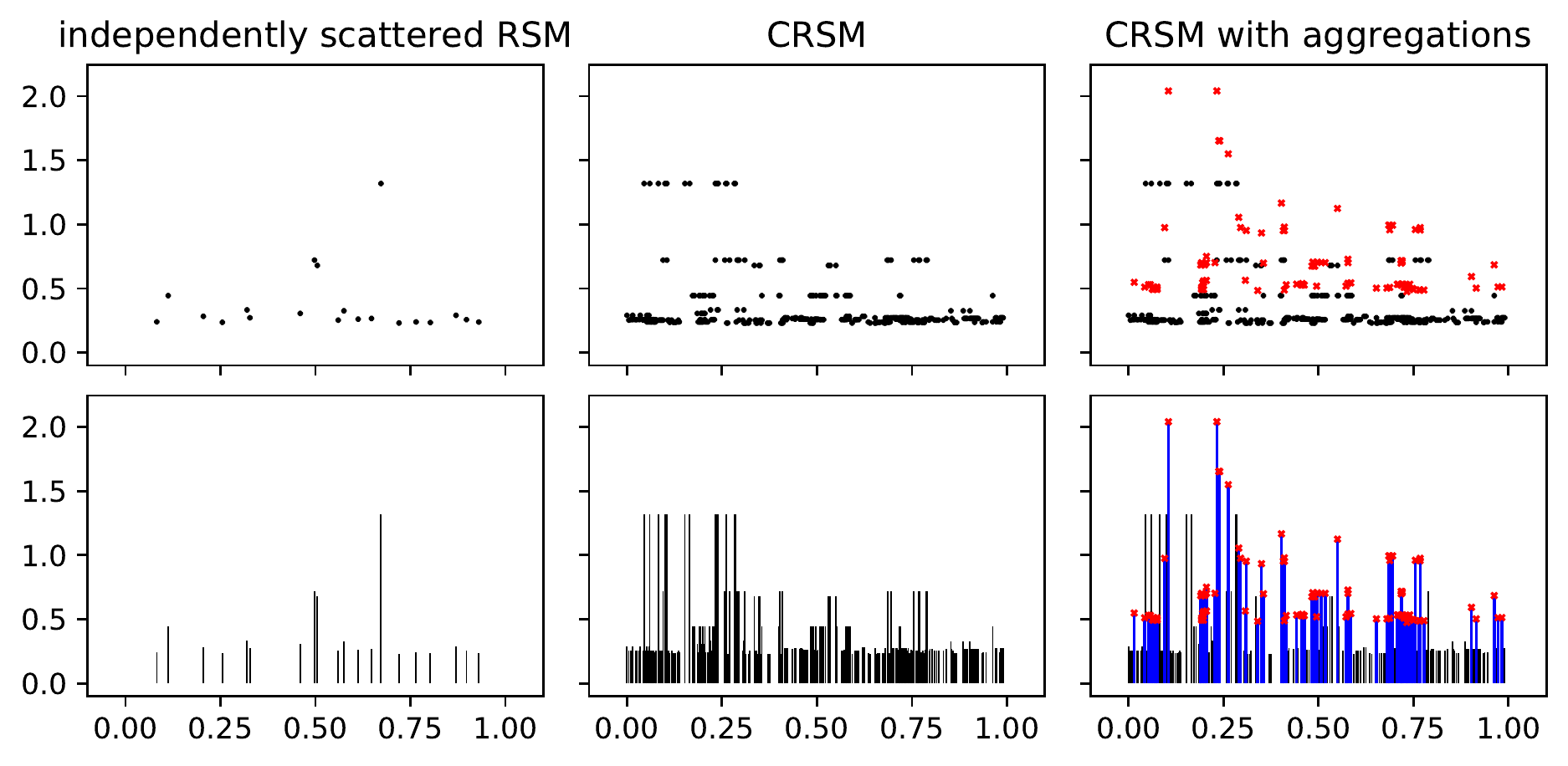}
\end{center}
\caption{Illustrations of random sup-measures. Top: the underlying point processes (magnitudes ($y$-axis) and locations ($x$-axis) of extremes). Bottom: the corresponding hypographs. 
Left: $\calM^{\rm is}_\alpha$, middle (and right): CRSM (with aggregations) with $\calR \eqd  \calR_{\beta}^{\rm srs}, \beta = 0.6$.
The top $m=20$ largest values $\{\Gamma_i^{-1/\alpha}\}_{i=1,\dots,m}$ and their locations are illustrated.
For CRSM, the extremes corresponding to intersections $\{\calR_i\cap\calR_j\}_{1\le i<j\le m}$ are marked by red crosses, and the corresponding `push-ups' of the hypographs are marked in blue color. In each case discrete random closed sets $R_{n,j}\subset\{0,\dots,n\}/n, j = 1,\dots,m$ are used (with $n=400$) to approximate the corresponding $\{\calR_j\}_{j=1,\dots,m}$ and their intersections.}
\label{fig:2}
\end{figure}
\begin{Rem}\label{rem:motivation}
There is a natural connection between Theorem \ref{thm:1} and \citep{samorodnitsky19extremal}. Assume $\alpha\in(0,2)$ and write 
\[
\zeta_{m,n}(k) := \summ j1{m}X_j\inddd{k\in R_{n,j}}.
\]
Under mild assumption and with $n\in\N$ fixed, after appropriate normalization $\{\zeta_{m,n}(k)\}_{k=0,\dots,n}$ converges in distribution as $m\to\infty$ to a stable (non-Gaussian) process, say $\{\calX_n(k)\}_{k=0,\dots,n}$. Then the relation between Theorem \ref{thm:1} and the limit theorem in \citep{samorodnitsky19extremal} can be summarized as the following diagram. \medskip
\begin{center}
\begin{tikzcd}[column sep=small]
\{\zeta_{m,n}(k)\}_{k=0,\dots,n} \arrow{rr}{\mbox{empirical  RSM, } m,n\to\infty \mbox{ (Thm.~\ref{thm:1})}} \arrow[swap]{dr}{m\to\infty} & & \CRSMa 
\\
& \mbox{stable process }  \{\calX_n(k)\}_{k=0,\dots,n}  \arrow[swap]{ur}{\mbox{empirical RSM, } n\to\infty, \mbox{ \citep{samorodnitsky19extremal}}}  & 
\end{tikzcd}  \medskip
\end{center}
The limit theorem in \citep{samorodnitsky19extremal} investigated the second part of the a {\em double-limit procedure} by first taking $m\to\infty$ so that $b_m\inv\{\xi_{m,n}(k)\}_{k\in\{0,\dots,n\}}\weakto \ccbb{\calX(k)}_{k=0,\dots,n}$ for some $b_m$,  and second taking  the limit as $n\to\infty$ to show $c_n\inv\max_{k/n\in\cdot}|\calX(k)| \weakto \calM_{\alpha,\calR}^{\rm Ca}$ for some $c_n$. % and  with a specific choice of $\calR = \calR_\beta^{\rm srs}$ in the limit (more precisely, the model of interest in \citep{samorodnitsky19extremal}) is $\{\calX(k)\}_{k\in\N}$ in the diagram above as the convergence from $\xi$ to $\calX$ is standard).
Here, Theorem \ref{thm:1} investigates the {\em single-limit} procedure by letting $m\to\infty$ and $n\to\infty$ at the same time, with $m = m_n$, and proving $\limn a_n\inv\max_{k/n\in\cdot}|\xi_{m_n,n}(k)|\weakto \calM_{\alpha,\calR}^{\rm Ca}$ and allows more general classes of $\calR$. 
%For the specific $\calR = \calR_{\alpha,\beta}^{\rm srs}$ that is used in \citep{samorodnitsky19extremal}, see  Section \ref{sec:srs}.

Comparisons between single-limit and double-limit procedures for aggregated models as in the diagram above have been known in the literature, especially for stochastic processes with long-range dependence. It is common that the same stochastic process arises in the limit for both procedures (e.g.~\citep{enriquez04simple,dombry09discrete}), and so is our case here. Therefore the aggregated model proposed in this paper provides an explanation to the abnormal limit behavior described in \citep{samorodnitsky19extremal}. In particular the aggregated model keeps two key features of the underlying dynamics: the renewal processes are having infinite mean renewal time, and the renewal times from independent renewal processes may intersect.  

We also point out that for more sophisticated models it is possible that, depending on the rate and also on the order of taking the limit in the double-limit procedure, different limit objects may arise at the end (e.g.~
\citep{pipiras04slow}).
\end{Rem}

\subsubsection*{Stable-regenerative sets with pinning} Here we present another variation based on stable-regenerative random sup-measures, based on {\em stable-regenerative sets with pinning}, denoted by  $\calR_{\beta}^{\rm srp}$. The CRSM with aggregations is then $\calM_{\alpha,\calR_\beta^{\rm srp}}^{\rm Ca}$, and we focus on $\calR_\beta^{\rm srp}$ in the sequel. Formally it has the conditional law of $\calR_\beta^{\rm sr}$ as before, given that $1\in\calR_\beta^{\rm sr}$. Our main reference is the very nice presentation of $\calR_\beta^{\rm srp}$, from a limit-theorem point of view, by \citet[Appendix A]{caravenna16continuum}. Let $\vv\tau$ be  the renewal process as in \eqref{eq:tau} with heavy-tailed return time satisfying \eqref{eq:Y}. Then, the following is from \citep[Proposition A.8]{caravenna16continuum}.
\begin{Prop}\label{prop:CSZ}
Under the assumption \eqref{eq:Y}, 
\equh\label{eq:rsp}
\calL\pp{\frac1n\pp{\vv\tau\cap\{0,\dots,n\}}\mmid n\in\vv\tau}\to \calL(\calR_\beta^{\rm srp}). 
\eque
\end{Prop}
The above is understood as that the conditional law of $R_n:=(\vv\tau\cap\{0,\dots,n\})/n$ given that $n\in\vv\tau$ converges weakly to the law of $\calR_\beta^{\rm srp}$. Note that actually more was proved  in \citep{caravenna16continuum}: it was first shown that the limit of the left-hand side exists (in the sense that the finite-dimensional distributions exist and are consistent).  It is also known that the so-obtained limit uniquely determines the law of a random closed set \citep[Proposition A.6]{caravenna16continuum}. So  $\calR_\beta^{\rm srp}$ can be defined via this scaling-limit approach.

Since $0$ and $1$ are fixed points so we choose $E = (0,1)$.  Now we examine the intersection property of $\calR_\beta^{\rm srp}$, which is very similar to the property of $\calR_\beta^{\rm sr}$. Since the simultaneous renewals yield a new renewal process, Proposition \ref{prop:CSZ} applies (see discussions after \eqref{eq:u}) and we have,
 \[
\calL\pp{R_{n,[\ell]}\mmid 1\in R_{n,[\ell]}}\to \calL(\calR_{\beta_\ell}^{\rm srp}),
\]
interpreted in a similar way as \eqref{eq:rsp}.
So to verify the assumption on intersections \eqref{eq:intersection_convergence}, it remains to show that,
\equh\label{eq:intersection_srp}
\bigcap_{i=1}^\ell \calR_{\beta,i}^{\rm srp} \eqd \calR_{\beta_\ell}^{\rm srp},
\eque
where $\{\calR_{\beta,i}^{\rm srp}\}_{i=1,\dots,\ell}$ are i.i.d.~copies of $\calR_\beta^{\rm srp}$. To see this, one way to go is to recall that (i) $\calR^{\rm sr}_\beta$ has the same law as the zero sets of a Bessel process starting at zero with dimension $2-2\beta$, and (ii) $\calR^{\rm srp}_\beta$ has the same law as the zero sets of a Bessel bridge (a Bessel process  that equal 0 at times 0 and 1, and  restricted to time interval $[0,1]$) with the same dimension. For the first fact, see \citep{bertoin99subordinators}. For the second, unable to find a reference we derive it again from \citep{caravenna16continuum}: it suffices to compare Eq.~(4.17) and Eq.~(4.18) therein (the finite-dimensional distributions of $\calR^{\rm sr}_\beta$ and $\calR^{\rm srp}_\beta$ respectively), and recall that the Bessel bridge, denoted by $\{\calB^{2-2\beta,\rm br}_t\}_{t\in[0,1]}$, can be obtained as a transformation of a Bessel process with the same dimension, denoted by $\{\calB_t^{2-2\beta}\}_{t\ge 0}$, via 
\[
\ccbb{\calB^{2-2\beta,\rm br}_t}_{t\in (0,1)}\eqd \ccbb{(1-t)^2\calB^{2-2\beta}_{t/(1-t)}}_{t\in(0,1)}.
\]
(See~\citep[Theorem 5.8]{pitman81bessel} and \citep[Section 5]{pitman82decomposition}.) Then \eqref{eq:intersection_srp} follows from the corresponding result for stable regenerative sets. This verifies \eqref{eq:intersection_convergence}.
 
Now, it remains to verify \eqref{eq:p_n} on $p_n(G)$. This time we have (recall that $R_n$ is considered with respect to the conditional probability)
\[
p_n(k) = \proba(k\in\vv\tau\mid n\in\vv\tau) = \frac{u(k)u(n-k)}{u(n)},
\]
with $u\in RV_{\beta-1}$ as in \eqref{eq:u}. Then, for $0<a<b<1$, Potter's bound implies that
 \[
 p_n((a,b)) \equiv \max_{k/n\in(a,b)}p_n(k) \le C \max_{k/n\in(a,b)}u(k)\le Cn^{\beta-1+\epsilon}
 \]
 for some $\epsilon>0$. It suffices to take $\epsilon\in(0,1-\beta)$
so that Assumption \ref{assump:Rn} is satisfied.

\subsection*{Acknowledgement} YW thanks Shuyang Bai, Olivier Durieu,  Ilya Molchanov, Gennady Samorodnitsky and Na Zhang for very helpful discussions, and the Associate Editor and two anonymous referees for helpful comments and suggestions. YW's research was partially supported by Army Research Office grants W911NF-17-1-0006 and 
W911NF-20-1-0139 at University of Cincinnati.

\def\cprime{$'$} \def\polhk#1{\setbox0=\hbox{#1}{\ooalign{\hidewidth
  \lower1.5ex\hbox{`}\hidewidth\crcr\unhbox0}}}
  \def\polhk#1{\setbox0=\hbox{#1}{\ooalign{\hidewidth
  \lower1.5ex\hbox{`}\hidewidth\crcr\unhbox0}}}

%\bibliographystyle{apalike}
%\bibliographystyle{imsart-nameyear}
%\bibliography{../../include/references,../../include/references18}
%\bibliography{references,references18}
\end{document}